\documentclass[11pt,reqno]{article}   
\usepackage{fullpage}

\usepackage[unicode=true]{hyperref}
\usepackage{amsmath}
\usepackage{cite}
\usepackage{amsfonts}
\usepackage{amssymb}
\usepackage{amsthm}
\usepackage{authblk}
\usepackage[pdftex]{color,graphicx}
\usepackage{adjustbox}
\usepackage{comment}
\usepackage{bbold}
\usepackage{tikz} 
\usetikzlibrary{decorations.pathreplacing,angles,quotes}
\usetikzlibrary{matrix,arrows,decorations.pathmorphing}
\usepackage{tikz-cd}
\usetikzlibrary{arrows}
\usetikzlibrary{decorations.markings}
\usepackage[all]{xypic}
\usepackage{bbold}
\usepackage{diagbox}
\usepackage{caption}
\usepackage{subcaption}
\usepackage{pdfpages} 
\usepackage{blkarray}
\usepackage{centernot}
\usepackage{mathtools}
\usepackage{stmaryrd}
\usepackage{soul}  
\usepackage{mypack-aziz} 

\definecolor{bluegray}{rgb}{0.4, 0.6, 0.8}
\definecolor{turquoise}{rgb}{0.2, 0.7, 0.6}

\newcommand{\Facc}{\operatorname{Face}}
\newcommand{\Facet}{\operatorname{Face}}
\newcommand{\Face}{\operatorname{Face}}

\usepackage{tikz-cd}

\begin{document}
	
	%
	%
	%

\title{Homotopical characterization of strongly contextual simplicial distributions on cone spaces}

	
	\author{Aziz Kharoof\footnote{aziz.kharoof@bilkent.edu.tr} }
	\author{Cihan Okay\footnote{cihan.okay@bilkent.edu.tr}}
	\affil{Department of Mathematics, Bilkent University, Ankara, Turkey}
	
	\maketitle

\begin{abstract}
This paper offers a novel homotopical characterization of strongly contextual simplicial distributions with binary outcomes, specifically those  
defined on the cone of a $1$-dimensional space. 
In the sheaf-theoretic framework, such distributions correspond to non-signaling distributions on measurement scenarios where each context contains $2$ measurements with binary outcomes. 
To establish our results, we employ a  homotopical approach that includes collapsing measurement spaces and introduce   
categories associated with simplicial distributions that can detect strong contextuality. 
\end{abstract}

\tableofcontents

\section{Introduction}
 
Simplicial distributions introduced in \cite{okay2022simplicial} provide a topological approach to the study of contextuality for collections of probability distributions extending the sheaf-theoretic framework of \cite{abramsky2011sheaf}.
A simplicial distribution is defined for a space of measurements and a space of outcomes represented by simplicial sets $X$ and $Y$, respectively.
In this paper, we study simplicial distributions for $2$-dimensional measurement spaces.
More specifically, we restrict our attention to spaces that can be obtained as the cone of a $1$-dimensional space.
Interesting examples include the well-known Clauser--Horne--Shimony--Holt (CHSH) scenario and the more general cycle scenarios \cite{e25081127} among many others.

 \begin{figure}[h!]
\centering
\begin{subfigure}{.33\textwidth}
  \centering
  \includegraphics[width=.8\linewidth]{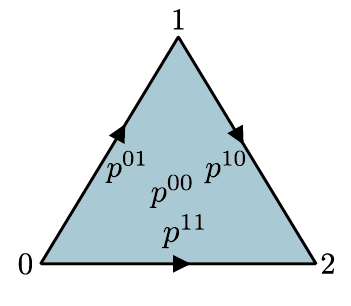}
  \caption{}
  \label{fig:triangle}
\end{subfigure}%
\begin{subfigure}{.33\textwidth}
  \centering
  \includegraphics[width=.8\linewidth]{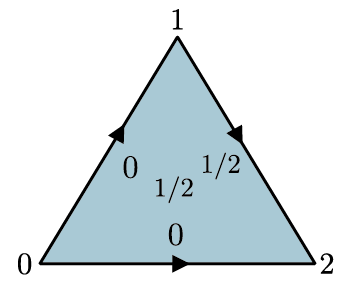}
  \caption{}
  \label{fig:PR1}
\end{subfigure}
\begin{subfigure}{.33\textwidth}
  \centering
  \includegraphics[width=.8\linewidth]{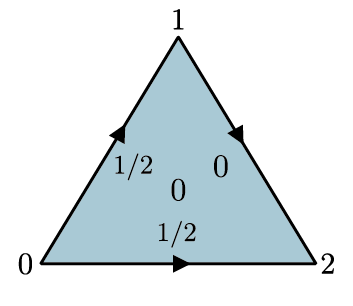}
  \caption{}
  \label{fig:PR2}
\end{subfigure}
\caption{ (a) A simplicial distribution on a triangle is specified by $p^{ab}\in [0,1]$ where  $a,b\in \ZZ_2$. {Popescu--Rohrlich boxes: $p_+$ (b) and $p_-$ (c).}
}
\label{fig:triangle-PR}
\end{figure}

Simplicial sets are combinatorial models of spaces that are more expressive than simplicial complexes.
A simplicial set $X$ consists of a sequence of sets $X_n$ representing the $n$-simplices together with face and the degeneracy maps relating these simplices of various dimensions.
A simplicial distribution on a pair $(X,Y)$ of simplicial sets is a simplicial set map 
$$
p:X\to D(Y)
$$
where $D$ is the distribution monad that replaces simplices with the set of probability distributions on those simplices. (We can work with the monad $D_R$ corresponding to $R$-valued distributions for an arbitrary semiring $R$.) 
Usually, $X$ comes with a finite set of {non-degenerate} simplices, and a simplicial distribution amounts to a collection of distributions $p_n(\sigma) \in D(Y_n)$  
for each {non-degenerate} simplex $\sigma$ of dimension $n$ related by {the simplicial structure maps.}
Given a simplicial set map $\varphi:X\to Y$ one can obtain a simplicial distribution by considering the composition $\delta^\varphi:X\xrightarrow{\varphi} Y\xrightarrow{\delta} D(Y)$  where $\delta$ sends a simplex to the distribution peaked at that simplex. Such distributions are called deterministic. 
A simplicial distribution is {non-}contextual if it lies in the image of the natural map:
$$
\Theta: D (\catsSet(X,Y)) \to \catsSet(X,D(Y)) 
$$
that sends a probabilistic mixture of simplicial set map $\varphi$ to a probabilistic mixture of $\delta^\varphi$. Otherwise, the distribution is called contextual. There is a notion of a support of a {given} simplicial distribution $p$ that consists of those simplicial set maps $\varphi:X\to Y$ such that the distribution $p_n(\sigma)$ evaluated at the simplex $\varphi_n(\sigma)$ is {non-zero}
for every {non-degenerate} simplex $\sigma$ of $X$. A simplicial distribution is called strongly contextual if the support is empty. Strongly contextual distributions naturally arise in quantum theory{; see  \cite{okay2022simplicial,okay2023equivariant}}. Our goal in this paper is to provide a homotopical characterization of strong contextuality.

For $2$-dimensional $X$ a typical simplicial set would consist of {non-degenerate} simplices of dimension $2$. 
For example, the $2$-simplex $\Delta[2]$ has a single {non-degenerate} simplex, and we can think of $X$ as obtained by gluing such simplices.
A simplicial distribution $p:\Delta[2]\to D(N\ZZ_2)$ is given by a distribution $p_n(\sigma) \in D(\ZZ_2^2)$; see Figure (\ref{fig:triangle}). 
We can think of $\Delta[2]$ as the join of the $0$-simplex $\Delta[0]$ and the $1$-simplex $\Delta[1]$, i.e., {it} is obtained by introducing a point and coning off. We will consider the  cone space of a $1$-dimensional simplicial set $X$, which is denoted by $\Delta[0]\ast X$.
A typical case is when the $1$-dimensional simplicial set is a circle $C$ consisting of $N$ edges. Then $\Delta[0]\ast C$ corresponds to the $N$-cycle scenario studied in \cite{e25081127}.
Our main result {(Corollary \ref{cor:homotopical characterization of strong contextuality})} is the following: 
 
\begin{thm*}
Let $X$ be a $1$-dimensional simplicial set.
A simplicial distribution 
$$
p:\Delta[0]\ast X\to D(N\ZZ_2)
$$
is strongly contextual if and only if there exists a circle $C\subset X$  
such that 
$$p|_{C}=\delta^{\varphi}$$ for some simplicial set map $\varphi:C\to N\ZZ_2$ that is not null-homotopic. 
\end{thm*}

An intermediate step in the proof of our result is the following characterization of strong contextuality (Theorem \ref{thm:SC1skelz2}): There exists a circle $C\subset {X}$ such that $p|_{\Delta[0]\ast C}$ is a Popescu--Rohrlich (PR) box; see Figure (\ref{fig:triangle-PR}).   
{To prove this result w}{e associate to {a} simplicial distribution a category whose morphisms correspond to conditional probability distributions. The composition in this category is induced by the $0$-order composition of the compository  {$D(\Delta_{\zz_d})$}
in {the sense of \cite{flori2013compositories}.}  
Similar kinds of categories are also studied in categorical probability theory  \cite{doberkat2003converse,panangaden1999category}.} 
{Other results we prove in this paper are as follows:}
\begin{itemize}
\item The cone construction has a closely related construction known as the d\' ecalage denoted by $\Dec^0(Y)$. We observe that there is a bijection of simplicial distributions
$$
\catsSet(\Delta[0]\ast X, D(Y)) \cong \catsSet(X,D(\Dec^0 Y))
$$ 
obtained by the well-known adjuction between the cone and the d\' ecalage functors (Section \ref{sec:decalage and the cone}). With this alternative perspective our main result gives a homotopical characterization of strong contextuality for simplicial distributions of the form $p:X\to D(\Dec^0N\ZZ_2)$.

\item There is a canonical map
$
d_0 : \Dec^0(Y) \to Y
$
which induces a map between simplicial distributions:
$$
\catsSet(X,D(\Dec^0Y)) \to \catsSet(X,DY)
$$
Given a deterministic distribution $\delta^\varphi$ associated to a simplicial set map $\varphi:X\to Y$, we study the preimage $\Facet(\varphi)$ of $\delta^\varphi$ under this map. Our main result  (Proposition \ref{pro:FaccCharac}) provides a description of $\Facet(\varphi)$ when {$X$ is a $1$-dimensional simplicial set and} $Y$ is the nerve space of a{n abelian} group $G$.

\item {We prove that $D(N\catC)$ is a compository for any small category $\catC$.  
In particular, $D(N\ZZ_d)$ is a compository {since $\Dec^0(N\zz_d)$ is isomorphic to the nerve of a category}. We use this observation to provide a conceptual underpinning for the composition in our associated category with a simplicial distribution.
}

\item Another ingredient in the proof of our main theorem {(Corollary \ref{cor:homotopical characterization of strong contextuality})} is the notion of a collapsing map $\pi:X\to \bar X$ that collapses a given $n$-simplex of $X$ to an $(n-1)$-simplex using a degeneracy map. Our collapsing result (Theorem \ref{thm:CollapThm}) essentially says that the properties of contextuality and being a vertex {(extremal simplicial distribution)} are preserved under the induced map:
$$
\pi^*: \catsSet(\bar X, D(Y)) \to \catsSet(X,D(Y))
$$

\item We give a characterization of strong contexuality for simplicial distributions 
$$p:X\to D_\BB({\Dec^0(N\ZZ_2)})$$ 
where {$\BB$ is the Boolean algebra $\set{0,1}$ and} $X$ is $1$-dimensional in terms of the associated logical category. More precisely, in Proposition \ref{Pro:CharofLSC}   we show that  strong contextuality of $p$ depends only on the endomorphisms of {the objects} in this category. 
\end{itemize}
The rest of the paper is organized as follows. In Section \ref{sec:simplicial distributions}, we introduce basic notions from convex sets, simplicial sets, and simplicial distributions. There are two important constructions: the d\' ecalage of a simplicial set and the cone of a simplicial set given in Section \ref{sec:decalage and the cone}. Then we move to simplicial distributions with the nerve space as the outcome space in Section \ref{sec:distributions on the nerve space} and relate them to distributions with outcome space the d\' ecalage of the nerve space.  
In Section \ref{sec:homotopical methods}, we introduce the collapsing technique and the  
{associated category to a simplicial distribution}. Section \ref{sec:logical categories} contains basic properties of {these} categories.
In Section \ref{sec:strong contextuality for binary outcomes} we {focus on the logical associated category to} prove our main theorem.
 
\paragraph{Acknowledgments.}
This work is supported by the Air Force Office of Scientific Research under
award number FA9550-21-1-0002.
{The second author also acknowledges support from the Digital Horizon Europe project FoQaCiA, GA no. 101070558.}

\section{Simplicial distributions}
\label{sec:simplicial distributions}

Simplicial distributions are first introduced in  \cite{okay2022simplicial}. {See also \cite{barbosa2023bundle} for a more general version.} In this section we recall the basic definitions following the more categorical approach given in \cite{kharoof2022simplicial}. 

\subsection{Convex sets}
Let $R$ be a commutative zero-sum-free semiring. This means that $a+b=0$ implies that $a=b=0$ for all $a,b\in R$.
The {\it distribution monad}  $D_R:\catSet \to \catSet$ is defined as follows:
	\begin{itemize}
		\item  {For a set $X$ the set $D_R(X)$ of $R$-distributions on $X$ consists of  functions $p:X\to R$ of finite support, i.e., $|\set{x\in X:\, p(x)\neq 0}|<\infty$, such that $\sum_{x\in X} p(x)=1$.}
		\item Given a function $f:X\to Y$ the function $D_R(f):D_R(X)\to D_R(Y)$ is defined by
		$$
		p \mapsto \left( y\mapsto \sum_{x\in f^{-1}(y)} p(x) \right). 
		$$ 
	\end{itemize}
	The structure maps of the monad are given as follows:
	\begin{itemize}
		\item  $\delta_X: X\to D_R(X)$ sends $x\in X$ to the delta distribution
		$$
		\delta^x(x') = \left\lbrace
		\begin{array}{ll}
			1 & x'=x,\\
			0 & \text{otherwise.}
		\end{array}
		\right.
		$$
		\item $\mu_X:D_R^2(X)\to D_R(X)$ sends a distribution $P$ to the distribution
		$$
		{\mu_X}(P)(x) = \sum_{p\in D_R(X)} P(p)p(x). 
		$$
	\end{itemize}
Algebras over the distribution monad are called {\it $R$-convex sets}. We will write 
$$
\nu^X:D_R(X) \to X
$$
for the structure map.
The category of $R$-convex sets will be denoted by $\catConv_R$. When $R=\RR_{\geq 0}$ we simply write $D$ for the distribution monad and $\catConv$ for the corresponding category.
There is an adjunction
\begin{equation}\label{eq:Set adjunction Conv}
D_R:\catSet \adjoint \catConv_R:U
\end{equation}
where $U$ is the forgetful functor and $D_R$ sends a set $X$ to the free convex set $D_R(X)$.

\begin{defn}\label{def:vertex}
{\rm
Let $(X,\nu^{X})$  be an $R$-convex set.
An element $x\in X$ is  
called a \emph{vertex}, or an \emph{extreme point}, if $x$ has a unique preimage under $\nu^X$.
}
\end{defn}
 
{A subset $U\subset X$ is called a {\it prime filter}  if} $U$ is an $R$-convex subset such that whenever a convex {combination} $\sum_{i=1}^n \alpha_i x_i$ of elements $x_i\in X$ lies in $U$ under the structure map $\nu^X$ we have $x_i\in U$ for all $1\leq i\leq n$ \cite[Definition 7]{jacobs2010convexity}.  

\begin{pro}\label{pro:ElemConv} 
Let $f:X\to Y$ be an $R$-convex map.
\begin{enumerate}
\item If $Z \subset Y$ is an $R$-convex subset, then $f^{-1}(Z)$ is an $R$-convex subset of $X$.
\item If $y\in Y$ is a vertex, then $f^{-1}(y)$ is a prime filter.
\item If $y\in Y$ is a vertex,
then every vertex of the $R$-convex subset $f^{-1}(y)$ is a vertex of $X$. 
\end{enumerate}
\end{pro}
\begin{proof}
Part (1): Given $Q \in D_R(f^{-1}(Z))$, we need to prove that $\nu^X(Q)\in f^{-1}(Z)$. For $y \notin Z$, we have
$$
D_R(f)(Q)(y)=\sum_{f(x)=y}Q(x)=0,
$$
that is, $D_R(f)(Q)\in D_R(Z)$. Therefore $\nu^Y(D_R(f)(Q))\in Z$ since $Z$ is a convex subset of $Y$. The map $f$ is in $\catConv_R$ therefore $f(\nu^X(Q))\in Z$, which implies that $\nu^X(Q)\in f^{-1}(Z)$.

Part (2):  
By part (1) we observe that $f^{-1}(y)$ is an $R$-convex subset of $X$. 
Given $Q
\in D_R(X)$ such that 
$\nu^{X}(Q) \in f^{-1}(y)$ we have $\nu^Y(D_R(f)(Q))=f(\nu^X(Q))=y$. Since $y$ is a vertex we have $D_R(f)(Q)=\delta^y$. Now, for $x \notin f^{-1}(y)$, 
$$
0=\delta^y(f(x))=
D_R(f)(Q)(f(x))=Q(x)+\sum_{x'\neq x: f(x')= f(x)} Q(x')
$$ 
which implies that $Q(x)=0$ since
$R$ is zero-sum-free.

Finally, part (3) follows directly from part (2).
\end{proof}

\subsection{Simplicial sets}

We recall some basic definitions from simplicial homotopy theory \cite{goerss2009simplicial}.	
The {\it simplex category} $\catDelta$ consists of
	\begin{itemize}
		\item the objects $[n]=\set{0,1,\cdots,n}$ for $n\geq 0$, and
		\item the morphisms $\theta:[m]\to [n]$ given by order preserving functions.
	\end{itemize}
A {\it simplicial set} is a functor $X:\catDelta^\op\to \catSet$. The set of $n$-simplices is usually denoted by $X_n=X([n])$. Alternatively, a simplicial set is a sequence of sets $X_0,X_1,\cdots$ with face maps ${d^X_i}:X_n \to X_{n-1}$ ($n \geq 1$, $0 \leq i \leq n$), and degeneracy maps ${s^X_j}:X_n \to X_{n+1}$ ($n \geq 0$, 0 $\leq j \leq n$) satisfying the simplicial identities (see \cite[Definition 3.2]{friedman2008elementary}). A simplex is called {\it degenerate} if it lies in the image of a degeneracy map, otherwise it is called {\it non-degenerate}. We write $X_n^\circ$ for the subset of non-degenerate $n$-simplices.
A non-degenerate simplex will be referred to as a {\it {generating} simplex} if it is not a face of another simplex.

An object $[m]$ in the simplex category gives a simplicial set $\Delta[m]$ whose set of $n$-simplices is given by  $\catDelta([n],[m])$. This simplicial set is called the {\it standard $m$-simplex}. We will write 
$$\sigma^{01\cdots m} \in \Delta[m]_m$$ 
for the simplex represented by the identity morphism $[m]\to [m]$. Note that all the other simplices can be obtained by applying a sequence of face and degeneracy maps. The horn space $\Lambda_i[n]$ is the simplicial subset of $\Delta[n]$ obtained by omitting the non-degenerate simplex $\sigma^{01\cdots n}$   and its $i$-th face. 

A {\it morphism of simplicial sets} is a natural transformation $f:X\to Y$ between the functors. In other words, it is a sequence of set maps $f_n:X_n \to Y_n$, where $n\geq 0$, compatible with the face and degeneracy maps. For a simplex $x \in X_n$, we will sometimes  write $f_x$ instead of $f_n(x)$.
A morphism $\theta$ in the simplex category induces a simplicial set map $\theta:\Delta[m]\to \Delta[n]$. 
We will write  $\catsSet$ for the category of simplicial sets, that is, the functor category $\catSet^{\Delta^\op}$.
This definition can be extended to an arbitrary category $\catC$ and the resulting functor category $\catC^{\Delta^\op}$ is denoted by $s\catC$.

\begin{defn}\label{def:Llinee}
{\rm 
A {\it line} $L$ 
is a simplicial set specified by 
a sequence of pairwise distinct $1$-simplices $\sigma_1,\cdots,\sigma_n\in L$ satisfying
$$
d_{i'_1}(\sigma_1)=d_{i_2}(\sigma_2)\, , \, 
d_{i'_2}(\sigma_2)=d_{i_3}(\sigma_3)\, ,\cdots,\, d_{i'_{n-1}}(\sigma_{n-1})=d_{i_n}(\sigma_n)
$$ 
where $i'_1,i_2,i'_2,i_3, \cdots,i'_{n-1},i_n
\in \{0,1\}$, and $i_j = 1-i'_j$ {for $2 \leq j \leq n-1$}. 
Let $i_1=1-i'_1$ and $i'_n=1-i_n$.
The $0$-simplices $x=d_{i_1}(\sigma_1)$ and $y=d_{i'_n}(\sigma_n)$ are called the {\it initial} and the {\it terminal} vertices of $L$. 
}
\end{defn}

Two $0$-simplices $x,y\in X_0$ are in the same path component if there exists a simplicial set map $h:L\to X$ with initial vertex $x$ and terminal vertex $y$. The set of path components is denoted by $\pi_0(X)$. When $\pi_0(X)$ is trivial   we say $X$ is {\it connected}.

\begin{defn}\label{def:simplicial homotopy}	{\rm	
Given two simplicial set maps $f,g:X\to Y$, we say that $f$ is {\it homotopic} to $g$ if there exists a homotopy,  a simplicial set map $H:X\times \Delta[1] \to Y$, such that $H|_{X\times \set{0}}=H\circ (\idy_X\times d^1)=f$ and $F|_{X\times \set{1}}=H \circ (\idy_X\times d^0)=g$. In this case we write $f\sim g$. For $Y$, a fibrant simplicial set (Kan complex), 
the set of homotopy classes of maps  
is denoted by $[X,Y]$.
We say $f$ is {\it null-homotopic} if it is homotopic to a constant map, i.e., $f\sim \ast$.	
}
\end{defn}

In this paper we will consider $n$-skeletal (or $n$-dimensional) simplicial sets where $n=1,2$. 
Let $\catDelta_{\leq n}$ denote the full subcategory of the simplex category consisting of objects $[m]$ where $0\leq m\leq n$. A simplicial set $X$ can be restricted to this subcategory. We will write $\tr_n(X):\catDelta_{\leq n}^\op\to \catSet$ for this restricted (truncated) simplicial set. This construction gives a functor $\tr_n:\catSet^{\Delta^\op} \to \catSet^{\Delta_{\leq n}^\op}$. There is an adjunction
\begin{equation}\label{eq:skeleton vs trancation adjunction}
\sk_n: \catSet^{\Delta^\op_{\leq n}} \adjoint \catSet^{\Delta^\op}:\tr_n
\end{equation}

\begin{defn}\label{def:skelatal}
{\rm
The {\it $n$-th skeleton} of a simplicial set $X$ is the simplicial set defined by
$$
\SK_n(X) = \sk_n\circ \tr_n(X).
$$
$X$ is said to be {\it $n$-skeletal} if the counit map $\SK_n(X)\to X$ of the adjunction in (\ref{eq:skeleton vs trancation adjunction}) is an isomorphism.}
\end{defn}

\subsection{Simplicial distributions}\label{susect:simp}
	  	
We can extend the distribution monad to a functor $D_R:\catsSet\to \catsSet$ by sending $X:\catDelta^\op\to \catSet$ to the composite $\catDelta^\op \xrightarrow{X}\catSet \xrightarrow{D_R} \catSet$. 
It turns out that this extended functor is also a monad; see \cite[Proposition 2.4]{kharoof2022simplicial}.  
Algebras {over this monad} are called {\it simplicial $R$-convex sets} and the resulting category is denoted by $s\catConv_R$.

	\begin{defn}\label{def:SimpDist}
		{\rm	
		A {\it simplicial scenario} is specified by a pair of simplicial sets $(X,Y)$ {where $X$ represents the space of measurements and $Y$ represents the space of outcomes}.
A {\it simplicial distribution} on a simplicial scenario $(X,Y)$ is a simplicial set map $p:X\to D_R(Y)$.	 
		}
	\end{defn}

\begin{pro}(\!\!\cite[Proposition 2.15]{kharoof2022simplicial})
\label{pro:Pro 2.15}
The functor $\St(-,-) : \St^{op} \times \St \to \catSet$ restricts to a functor
$$
\St(-,-): \St^{op} \times s\catConv_R \to \catConv_R
$$
\end{pro}

The structure map $\nu$ of the simplicial $R$-convex set $\catsSet(X,D_R(Y))$ is defined by  
	\begin{equation}\label{piForm}
	\nu(Q)_n(x)=\sum_{p \in \catsSet(X,D_R(Y)) }Q(p)\,p_n(x)
	\end{equation}
where $Q \in D_R(\catsSet(X,D_R(Y)))$ and $x\in X_n$.
When $R=\RR_{\geq 0}$ this gives the usual convex structure. Moreover, if $X$ has finitely many non-degenerate simplices then $\catsSet(X,D(Y))$ is a convex polytope with finitely many vertices; see for example \cite{okay2022mermin,e25081127}.

An important class of simplicial distributions are deterministic distributions obtained as the image of the following map
$$
(\delta_Y)_* : \catsSet(X,Y) \to \catsSet(X,D_R(Y))
$$	
For a simplicial set map $\varphi:X\to Y$ we write $\delta^\varphi=(\delta_Y)_*(\varphi)$ for the corresponding deterministic distribution.

\begin{pro}(\!\!\cite[Proposition 5.14]{kharoof2022simplicial})
\label{pro:Pro 5.14}
Assume that the semiring $R$ is also integral ($a\cdot b=0$ implies $a=0$ or $b=0$ for all $a,b\in R$). 
Every deterministic distribution in $\St(X,D_R(Y))$ is a vertex.
\end{pro}
	
The transpose of $(\delta_Y)_*$ under the adjunction in (\ref{eq:Set adjunction Conv}) gives a comparison map 
	\begin{equation}\label{eq:Theta}
\Theta_{X,Y}: D_R(\catsSet(X,Y)) \to \catsSet(X,D_R(Y))
	\end{equation}

\begin{defn}
	\label{def:contextual-morphism}	
		{\rm
A simplicial distribution $p:X\to D_RY$ is called {\it non-contextual} if $p$  lies in the image of $\Theta_{X,Y}$. Otherwise, it is called {\it contextual}.  
		}
\end{defn}

\begin{example}\label{ex:DeltaClass}
{\rm
Given a simplicial set $Y$, every simplicial distribution on the simplicial scenario $(\Delta[n],Y)$ is non-contextual {since $\Theta$ is an isomorphism.} 
} 
\end{example}	
\begin{pro}\label{lem:Exten}
Given a simplicial set map $f:X_1 \to X_2 $, if  
$p \in \catsSet(X_2,D_R(Y))$ is non-contextual then  $f^\ast (p)$ is also non-contextual. 
\end{pro}
\begin{proof}
Follows directly from the commutative diagram
\begin{equation}
\label{diag:ImpDiag}
\begin{tikzcd}[column sep=huge,row sep=large]
D_R(\catsSet(X_2,Y))  
\arrow[rr,"
\Theta_{X_2,Y}"]
 \arrow[d,"D_R(f^{\ast})"'] && \catsSet(X_2,D_R(Y))  
 \arrow[d,"f^\ast"] \\
D_R(\catsSet(X_1,Y)) 
 \arrow[rr,"\Theta_{ X_1,Y}"] && \catsSet(X_1,D_R(Y)) 
\end{tikzcd}
\end{equation}
\end{proof}
Similarly, we have the following:
\begin{pro}\label{lem:Exten2}
Given a simplicial set map $f:Y_1 \to Y_2 $, if $p \in \catsSet(X,D_R(Y_1))$ is   non-contextual then $f_\ast (p)$ is also non-contextual. 
\end{pro}

The {\it support of a simplicial distribution}
	$p:X\to D_R Y$ is defined by
	$$
	\supp(p) =\{ \varphi \in \St(X ,Y)  :\, p_n(x)(\varphi_n(x)) \neq 0,\; \forall x\in X_n,\, n \geq 0  \}.
	$$
Throughout the paper we will restrict to $X$ with finitely many simplices. Such simplicial sets are determined by a finite number of {generating} simplices.
Note that to have $\varphi\in \supp(p)$ it suffices that 
$p_{\sigma}(\varphi_{\sigma})\neq 0$ for the 
{generating} simplices $\sigma$ of $X$.

\begin{defn}\label{def:strong contextuality}
{\rm
We say $p:X\to D_R(Y)$ is {\it strongly contextual} if $\supp(p)$ is empty.
}
\end{defn}

\begin{pro}(\!\!\cite[Proposition 2.12]{kharoof2022simplicial})
\label{pro:Pro 2.12}
If a simplicial distribution 
$p:X\to D_R(Y)$ is strongly contextual 
then it is contextual.
\end{pro}
%

If the outcome {space} $Y$ is a simplicial group, then $\catsSet(X,D_R(Y))$ is a monoid, where the product is defined as follows: 
$$
(p\cdot q)_x(y)=\sum_{{y_1\cdot y_2}=y}p_x(y_1)q_x(y_2)
$$
%
for $p,q\in \catsSet(X,D_R(Y))$, $x\in X_n$, and $y\in Y_n$. In fact, the set $\catsSet(X,D_R(Y))$ is an $R$-convex monoid (see \cite[Corollary 5.2]{kharoof2022simplicial}). Moreover, the group $\catSet(X,Y)$ acts on $\catsSet(X,D_R(Y))$ by the following rule:
$$
\varphi \cdot p=\delta^{\varphi}\cdot p
$$
where $\varphi \in \catsSet(X,Y)$ and $p\in \catsSet(X,D_R(Y))$. 
\begin{prop}\label{pro:SCGroupaction}
Let $Y$ be a simplicial group. For
$p \in \catsSet(X,D_R(Y))$ and a map $\varphi \in \catsSet(X,Y)$,
the simplicial distribution $p$ is strongly contextual if and only if $\delta^{\varphi}\cdot p$ is strongly contextual. 
\end{prop}
\proof{
Given $\psi \in \catsSet(X,Y)$ and $x \in X_n$, we have  
%
$$
(\delta^{\varphi}\cdot p)_x\left(({\varphi\cdot \psi})_x\right)=\sum_{{y_1y_2=
\varphi_x\psi_x}}\delta^{\varphi_x}(y_1)p_x(y_2)=p_x(\psi_x) .
$$
We conclude that $\psi\in \supp(p)$ if and only if ${\varphi\cdot\psi} \in \supp(\delta^\varphi \cdot p)$.

\subsection{D\' ecalage and the cone constructions}
\label{sec:decalage and the cone}

\begin{defn}{\rm (\!\!\cite{stevenson2011d})
\label{def:decalage} 
The {\it d\' ecalage} of a simplicial set $X$ is the simplicial set   $\Dec^0 (X)$ obtained by shifting the simplices of $X$ down by one degree, i.e., $(\Dec^0 X)_n = X_{n+1}$, and forgetting the first face and degeneracy maps.
}\end{defn}

The d\' ecalage construction gives a functor $\Dec^0:\catsSet\to \catsSet$. It has a left adjoint given by the cone construction.

\begin{defn}\label{def:cone}
{\rm
The cone of a simplicial set $X$ is the simplicial set $\Delta[0]\ast X$ defined by  
\begin{itemize}
\item $(\Delta[0]\ast X)_n = \set{c_n} \sqcup X_n \sqcup \left( \sqcup_{k+1+l=n} \set{c_k} \times X_l \right)$ where $c_n=(s_0)^n(\ast)$.  
\item For $(c_k,\sigma)\in \set{c_k}\times X_l$ 

$$
d_i(c_k,\sigma) = \left\lbrace
\begin{array}{ll}
(c_{k-1},\sigma) & i\leq k\\
(c_k,d_{i-1-k}\sigma) & i>k,
\end{array}
\right.
$$
{where $(c_{-1},\sigma)=\sigma$},
and
$$
s_j(c_k,\sigma) = \left\lbrace
\begin{array}{ll}
(c_{k+1},\sigma) & j\leq k\\
(c_k,s_{{j}-1-k}\sigma) & {j}>k.
\end{array}
\right.
$$
Otherwise,  the face and the degeneracy maps on the $\set{c_n}$ and $X_n$ factors act the same as in $\Delta[0]$ and $X$.
\end{itemize} 
}
\end{defn}
 
The cone construction is also functorial. Moreover, there is an adjunction 
\begin{equation}\label{eq:Cone vs declage adjunction}
\Delta[0]\ast (-): \catsSet \adjoint \catsSet: \Dec^0
\end{equation}
$\Dec^0(X)$ is an augmented simplicial set with the augmentation map 
\begin{equation}\label{eq:decalage d0}
d_0: \Dec^0(X) \to X_0
\end{equation} 
which is a deformation retract \cite[Lemma 2.6]{stevenson2011d}.

\begin{lemma}\label{lem:transcomm}
Let $\varphi:X \to Dec^0(Y)$ be a simplicial set map, and let $\varphi':\Delta[0] \ast X \to Y$ be the transpose of $\varphi$ under the adjunction in (\ref{eq:Cone vs declage adjunction}). The following diagram commutes:
$$
\begin{tikzcd}[column sep=huge,row sep=large]
X
\arrow[r,"\varphi"]
 \arrow[d,hook,""'] &
 Dec^0(Y)
 \arrow[d,"d_0"] \\
\Delta[0] \ast X 
 \arrow[r,"\varphi'"] & Y 
\end{tikzcd}
$$ 
\end{lemma}

The notions of contextuality for the simplicial scenarios $(\Delta[0]\ast X,Y)$ and $(X,\Dec^0(Y))$ coincide:

\begin{pro}(\!\!\cite[Proposition 2.23]{kharoof2022simplicial})
\label{pro:Pro 2.23} 
The adjunction given in  (\ref{eq:Cone vs declage adjunction}) induces  
a commutative diagram of $R$-convex sets
$$
\begin{tikzcd}[column sep=huge,row sep=large]
D_R(\catsSet(\Delta[0] \ast X,Y )) 
\arrow[rr,"\Theta_{\Delta[0] \ast X,Y}"]
 \arrow[d,"\cong"] && \catsSet(\Delta[0] \ast X,D_R(Y)) \arrow[d,"\cong"] \\
D_R(\catsSet(X,\Dec^0Y)) \arrow[rr,"\Theta_{X,\Dec^0(Y)}"] && \catsSet(X,D_R(\Dec^0Y))
\end{tikzcd}
$$ 
\end{pro}

\begin{defn}\label{FacdefNH}
{\rm
Given $\varphi \in \catsSet(X,Y)$ we define the {\it face} $\Facc(\varphi)$ at $\varphi$ to be the preimage of 
$\delta^\varphi$
under the map 
$$
D_R(d_0)_\ast:\catsSet(X,D_R(\Dec^0Y)) \to \catsSet(X,D_R(Y))
$$ 
}
\end{defn}

In other words, a simplicial distribution  $p : X \to D_R(\Dec^0(Y))$ belongs to $\Facc(\varphi)$ if and only if it makes the following diagram commute
$$
\begin{tikzcd}[column sep=huge,row sep=large]
 & D_R(\Dec^0(Y)) 
\arrow[d,"D_R(d_0)"]  \\
X \arrow[ru,"p"]  \arrow[r,"\delta^{\varphi}"']  & D_R(Y)
\end{tikzcd}
$$
We will denote the set of the vertices of $\Facc(\varphi)$ by $V(\Facc(\varphi))$ (Definition \ref{def:vertex}).

\begin{pro}\label{pro:FaccPrimeFill}
Given a simplicial set map $\varphi:X\to Y$, let $\Face(\varphi)$ denote the corresponding face (Definition \ref{FacdefNH}).
\begin{enumerate}
\item Every element of $V(\Facc(\varphi))$ is a vertex of $\catsSet(X,D_R(\Dec^0Y))$.
\item If $Y$ is a simplicial group, then for $p \in \Facc(\varphi)$ and $q \in \Facc(\psi)$, we have $p \cdot q \in 
\Facc(\varphi \cdot \psi)$.
\item The face $\Facc(\varphi)$ is a prime filter.
\end{enumerate}
\end{pro}         
\begin{proof}
According to Proposition \ref{pro:Pro 2.15} 
the map $D_R(d_0)_\ast$ is in $\catConv_R$. Part (1) follows from part (3) of Proposition \ref{pro:ElemConv}, and part (2) follows from the fact that $D_R(d_0)_\ast$ is a homomorphism of monoids.
Finally, by part (2) of Proposition \ref{pro:ElemConv} we {obtain} part (3).
\end{proof}
\begin{pro}\label{NotnullSC}
If $\varphi:X\to Y$ is not null-homotopic, then every simplicial distribution in 
$\Facc(\varphi)$ is strongly contextual. 
\end{pro}
\begin{proof}
Let $p \in \Facc(\varphi)$. 
If $\psi \in \supp(p)$, then for $x \in X_n$ we have
$$
\begin{aligned}
(\delta^{\varphi})_n(x)\left(\left((d_0)_\ast(\psi)\right)_n(x)\right)&=D_R(d_0)_\ast(p)_n(x)\left(\left((d_0)_\ast(\psi)\right)_n(x)\right)\\
&=D_R((d_0)_n)(p_n(x))\left((d_0)_n(\psi_n(x))\right)\\
&=\sum_{y: \, (d_0)_n(y)=
(d_0)_n(\psi_n(x))}p_n(x)(y)\\
&=p_n(x)(\psi_n(x))+ \cdots 
\end{aligned}
$$
The last sum is not zero, because $p_n(x)(\psi_n(x))>0$ and $R$ is zero-sum-free. Therefore $(d_0)_\ast(\psi) \in \supp(\delta^\varphi)=\{\varphi\}$. We conclude that $(d_0)_\ast(\psi)=\varphi$, which means that $\varphi$ is null-homotopic. 
\end{proof}

\section{Distributions on the nerve space}
\label{sec:distributions on the nerve space} 
 

Throughout this section $G$ denotes a finite group and $X$ denotes a simplicial set with $X_0$ and $X_1$ are both finite.

\subsection{Nerve space}

{Let $NG$ denote the nerve of the group $G$; see for example \cite{goerss2009simplicial}.}
We prove some results on the homotopy classes of maps $X\to NG$.

\begin{lemma}\label{lem:Kerofd0} 
{Let $X$ be a connected simplicial set.} The group of null-homotopic maps in $\catsSet(X,NG)$ is isomorphic to $G^{|X_0|-1}$.
\end{lemma}
\begin{proof}
The group of the null-homotopic maps in $\catsSet(X,NG)$ is exactly the image of 
$(d_0)_\ast: \catsSet(X,\Dec^0(NG)) \to \catsSet(X,NG)$. We have $\catsSet(X,\Dec^0(NG)) \cong 
\catSet(X_0,G)=G^{|X_0|}$. Now, we characterize the maps in the kernel of $(d_0)_\ast$.
Suppose that $\varphi \in \Ker{(d_0)_\ast}$, then for $x \in X_1$ there is $g \in G$ such that 
$\varphi_1(x)=(g,1_G)$. Therefore, we have
$$
\varphi_0(d^X_0(x))=
d^{\Dec^0(NG)}_0(\varphi_1(x))=g\cdot 1_G =g
$$ 
and  
$$
\varphi_0(d^X_1(x))=
d^{\Dec^0(NG)}_1(\varphi_1(x)) =g.
$$ 
Since $X$ is connected, we obtain that $\varphi_0(z)=g$ for every $z\in X_0$. On the other hand, every $g \in G$ gives us a simplicial set map 
$\varphi: X \to \Dec^0(NG)$ by setting $\varphi_0(z)=g$ for every $z\in X_0$. In this case $\varphi_n(x)=(g,1_G,\cdots,1_G)$ for every $x \in X_n$, thus $\varphi \in \Ker (d_0)_{\ast}$. We conclude that $\Image (d_0)_\ast \cong G^{|X_0|}/\{(g,\cdots,g):\, g \in G\} \cong G^{|X_0|-1}$.  
\end{proof}

\begin{pro}\label{pro:number of not null-homotopic maps NG}
Let $X$ be a $1$-skeletal {connected} simplicial set, 
$n=|X_0|$, and  $m$ denote the number of non-degenerate simplices in $X_1$. 
Then the number of maps in $\catsSet(X,NG)$ that are not null-homotopic is given by
$|G|^m-|G|^{n-1}$. 
%
\end{pro}
\begin{proof}
Since $X$ is $1$-skeletal, we have $\catsSet(X,NG)\cong G^m$. Therefore by  Lemma \ref{lem:Kerofd0} we obtain that the number of 
maps in $\catsSet(X,NG)$ that are not null-homotopic is 
$|G|^m-|G|^{n-1}$. 

\end{proof}
%


\begin{defn}\label{def:circle}
{\rm
Let $C$ be a line (Definition \ref{def:Llinee}).
If in addition $d_{i_1}(\sigma_1)=d_{i'_n}(\sigma_n)$ (the initial and the terminal vertices coincide) then the line is called a {\it circle}.
}
\end{defn}

%
%


Given a line $L$ and a simplicial set map 
$\varphi : L \to NG$ 
we define 
\begin{equation}\label{eq:C varphi}
(L,\varphi) = \prod_{i=1}^{n}g_i
\end{equation}
 where 
\begin{equation}\label{Eq:gkvarph}
g_k =\begin{cases}
\varphi_{\sigma_k}  & \text{if} \; \; \; i_k=1\\
\varphi_{\sigma_k}^{-1}  & \text{if} \; \;\; i_k=0
\end{cases}
\end{equation}
for $1 \leq k \leq n$. 

%
\begin{pro}\label{pro:NullHomCirc}
For a circle $C$, a simplicial set map $\varphi:C \to NG$ is null-homotopic if and only if $(C,\varphi)=1_G$.
\end{pro}
\begin{proof} 
Suppose there is  a simplicial set map $\psi:\Delta[0] \ast C \to NG$ such that $\psi|_C=\varphi$. For $1 \leq k \leq n$, the edges of the triangle $(c_0,\sigma_k)$ give us the following equations 
$$
\begin{aligned}
\psi_{(c_0,d_{i_k}(\sigma_k))} \cdot \psi_{\sigma_k}=\psi_{(c_0,d_{i'_k}(\sigma_k))} & \text{ if} \,\,\, i_k=1,\\
\psi_{(c_0,d_{i'_k}(\sigma_k))} \cdot \psi_{\sigma_k}=\psi_{(c_0,d_{i_k}(\sigma_k))} & \text{ if} \,\,\, i_k=0.
\end{aligned}
$$
Note that $\psi_{\sigma_k}=\varphi_{\sigma_k}$, so we obtain that $\psi_{(c_0,d_{i_k}(\sigma_k))} \cdot g_k=\psi_{(c_0,d_{i'_k}(\sigma_k))}$ (see Equation (\ref{Eq:gkvarph})). Using these equations and the way that the edges $\sigma_1,\cdots,\sigma_n$ are glued together, we obtain that 
$$
\begin{aligned}
\psi_{(c_0,d_{i_1}(\sigma_1))}&=\psi_{(c_0,d_{i'_n}(\sigma_n))}\\
&=\psi_{(c_0,d_{i_n}(\sigma_n))} \cdot g_{n}\\
&=\psi_{(c_0,d_{i'_{n-1}}(\sigma_{n-1}))} \cdot g_n\\
&=\psi_{(c_0,d_{i_{n-1}}(\sigma_{n-1}))} \cdot g_{n-1}\cdot g_n
\\
&= \cdots =
\psi_{(c_0,d_{i_{1}}(\sigma_{1}))} \cdot \prod_{i=1}^{n}g_i.
\end{aligned}
$$
This implies that $\prod_{i=1}^{n}g_i=1_G$. Now, suppose conversely that $\prod_{i=1}^{n}g_i=1$. We define a null-homotopy $\psi: \Delta[0] \ast C \to NG$ 
%
by
$$
\psi_{(c_0,d_{i_k}(\sigma_k))}=\prod_{i=1}^{k-1}g_i
$$
where $1 \leq k \leq n$ and the empty product is by convention defined to be $1_G$. 
The map $\psi$ is well-defined because a map in $\catsSet(\Delta[2],NG)$ is uniquely defined by the restriction on any horn $\Lambda_i[2]$. It remains to show that $\psi|_{C}=\varphi$. For 
$1 \leq k \leq n-1$ if $i_k=1$ then 
$(\prod_{i=1}^{k-1}g_i) \cdot \psi_{\sigma_k}=\prod_{i=1}^{k}g_i$, so $\psi_{\sigma_k}=g_k=\varphi_{\sigma_k}$. If $i_k=0$ then $(\prod_{i=1}^{k}g_i) \cdot \psi_{\sigma_k}=\prod_{i=1}^{k-1}g_i$, so $\psi_{\sigma_k}={g_k}^{-1}=\varphi_{\sigma_k}$. For $k=n$, if $i_n=1$ then $(\prod_{i=1}^{n-1}g_i) \cdot \psi_{\sigma_n}=1_G$. Because $\prod_{i=1}^{n}g_i=1_G$  we obtain that $\psi_{\sigma_n}=g_n=\varphi_{\sigma_n}$.
Finally, if $i_n=0$ then $1_G \cdot \psi_{\sigma_n}=\prod_{i=1}^{n-1}g_i$. Again because $\prod_{i=1}^{n}g_i=1_G$  we obtain that $\psi_{\sigma_n}={g_n}^{-1}=\varphi_{\sigma_n}$. 
\end{proof}
\begin{pro}\label{pro:Whomoto}
Let $X$ be a
$1$-skeletal simplicial set.
A simplicial set map  $\varphi \in \catsSet(X,NG)$
is null-homotopic if and only if for every circle $C\subset X$ the restriction
$\varphi|_C$ is null-homotopic.
\end{pro}
\begin{proof}
There exists  circles  $C_1,C_2,\cdots,C_m$ in $X$ such that 
$X$ is homotopic to $C_1 \vee C_2 \vee \cdots \vee C_m$. 
Therefore we have a map 
$$
\catsSet(X,NG)\to \prod_{i=1}^{m}\catsSet(C_i,NG)\cong 
\catsSet(\vee_{i=1}^m C_i,NG)
$$
that sends 
$\varphi$ to $\{\varphi|_{C_i}\}
$. This map
induces an isomorphism $[X,NG]\cong \prod_{i=1}^{m}[C_i,NG]$.
Therefore $\varphi$ is null-homotopic if and only if $\varphi|_{C_i}$ is null-homotopic for every 
$1 \leq i \leq m$. This happens if and only if $\varphi|_{C}$ for every circle $C\subset X$.
\end{proof}
\begin{corollary}\label{cor:NullHomAct}
Let $X$ be a $1$-skeletal simplicial set.
A simplicial set map $\varphi:X \to NG$ is null-homotopic if and only if $(C,\varphi|_{C})=1_G$ for every circle $C\subset X$. 
\end{corollary}
\begin{proof}
Follows from Propositions \ref{pro:NullHomCirc} and \ref{pro:Whomoto}.
\end{proof}

\subsection{D\'ecalage of the nerve space}

Let $M$ be a finite monoid.
For $n\geq 0$ we define
a map 
\begin{equation}\label{eq:SecDef}
T_n : D(M^n) \to D(M^{n+1})
\end{equation}
by 
$$T_n(p)(m_1,\cdots,m_{n+1})=\frac{p(m_2,\cdots,m_{n+1})}{|M|}$$ 
where $p \in D(M^n)$ and $(m_1,\cdots,m_{n+1})\in M^{n+1}$. 
%
%
\begin{lem}\label{TheSeccc}
The maps {$\{T_n\}_{n\geq 0}$} in Equation
(\ref{eq:SecDef}) assemble into a simplicial set map 
$$T:D(NM) \to D(\Dec^0(NM))$$ 
that gives a section of $D(d_0) : D(\Dec^0(NM)) \to D(NM)$.
\end{lem}
\begin{proof}
Given $0 \leq i < n$, $x_1,\cdots,x_{n+2} \in M$, and $p \in D(M^n)$, 
%
we have {
$$
\begin{aligned}
D(d_{i+1})(T_n(p))(x_1,\cdots,x_n) &= 
\sum_{d_{i+1}(y_1,\cdots y_{n+1})=(x_1,\cdots,x_n)}
T_n(p)(y_1,\cdots,y_{n+1})
\\
&=\sum_{y_{i+1} \cdot y_{i+2}=x_{i+1}}T_n(p)(x_1,\cdots,
x_i,y_{i+1},y_{i+2},x_{i+2},\cdots,x_n) \\
&=\sum_{y_{i+1} \cdot y_{i+2}=x_{i+1}}\frac{   p(x_2,\cdots,x_i,
y_{i+1}, y_{i+2},x_{i+2},\cdots,x_{n})}{|M|}\\
&=     \frac{\sum_{d_i(y_1,\cdots,y_n)=(x_2,\cdots,x_n)} p(y_1,\cdots,y_n)}{|M|}\\
&= 
\frac{  D(d_i)(p)(x_2,\cdots,x_{n})}{|M|}\\
&=T_{n-1}(D(d_{i})(p))(x_1,\cdots,x_n).
\end{aligned}
$$}
%
%
For $i=n$, we have 
$$
\begin{aligned}
D(d_{n+1})(T_n(p))(x_1,\cdots,x_n) &= 
\sum_{y_{n+1}\in M}
T_n(p)(x_1,\cdots,x_n,y_{n+1})
\\
&=\sum_{y_{n+1}\in M}
\frac{p(x_2,\cdots,x_n,y_{n+1})}{|M|} \\
&=\frac{\sum_{y_{n+1}\in M}
p(x_2,\cdots,x_n,y_{n+1})}{|M|} \\
& =\frac{D(d_{n})(p)(x_2,\cdots,x_n)}{|M|} \\
&=T_{n-1}(D(d_{n})(p))(x_1,\cdots,x_n) .
\end{aligned}
$$
Now, if $x_{i+2}=1_M$, then 
$$
\begin{aligned}
D(s_{i+1})(T_n(p))(x_1,\cdots,x_{n+2})&= 
T_n(p)(x_1,\cdots,x_{i+1},x_{i+3},\cdots,x_{n+2})\\
&=\frac{p(x_2,\cdots,x_{i+1},x_{i+3},\cdots,x_{n+2})}{|M|}
\\
&=
\frac{D(s_i)(p)(x_2,\cdots,x_{n+2})}{|M|}
\\
&=T_{n+1}(D(s_i)(p))(x_1,\cdots,x_{n+2})
\end{aligned}
$$
and if $x_{i+2} \neq 1_M$, then  $\left(D(s_{i+1})(T_n(p))\right)(x_1,\cdots,x_{n+2})
=T_{n+1}(D(s_i)(p))(x_1,\cdots,x_{n+2})=0$.
It remains to prove that 
$D(d_0)\circ T = \Id_{D(NM)}$. 
We have
$$
D(d_0)(T_n(p))(x_1,\cdots,x_n)=
\sum_{y \in M}T_n(p)(y,x_1,\cdots,x_n)= 
\sum_{y \in M}
\frac{p(x_1,\cdots,x_n)}{|M|} =p(x_1,\cdots,x_n).
$$
\end{proof}

{Using the $T$ map we can prove the following contractibility result.}

\begin{pro}\label{pro: DNGContr}
The simplicial sets $D(N{M})$ and $D(\Dec^0(N{M}))$ are both contractible.  
\end{pro}
\begin{proof} 
The simplicial set $D(N{M})$ is reduced, so by  \cite[Lemma 2.6]{stevenson2011d} we {obtain} that 
$Dec^0(D(N{M}))=D(Dec^0(N{M}))$ is contractible. {Therefore} by Lemma
\ref{TheSeccc} 
%
%
the map $\Id_{D(N{M})}$ is nullhomotopic, which means that $D(N{M})$ is contractible. 
\end{proof}

\begin{pro}\label{pro:TConv}
The map $T: D(NM)\to D(\Dec^0(NM))$ belongs to  
$s\catConv$.
\end{pro}
\begin{proof}
We need to prove that $T_n$ is a convex map,
 for every $n\geq 0$. Given $p,q \in D(M^n)$, $\alpha \in [0,1]$, and $m_1,\cdots, m_{n+1} \in M$, 
 we have
$$ 
\begin{aligned}
T_n(({\alpha}p + (1-\alpha)q))(m_1,\cdots,m_{n+1})&=
 \frac{({\alpha}p + (1-\alpha)q)(m_2,\cdots,m_{n+1})}{|M|} \\
&= \alpha\frac{p(m_2,\cdots,m_{n+1})}{|M|}+
(1-\alpha)\frac{q(m_2,\cdots,m_{n+1})}{|M|}
\\
&= \alpha T_n(p)(m_1,\cdots,m_{n+1})+
(1-\alpha)T_n(q)(m_1,\cdots,m_{n+1})\\
&= (\alpha T_n(p)+
(1-\alpha)T_n(q))(m_1,\cdots,m_{n+1}).
\end{aligned}
$$
\end{proof}

{Next, we turn to applications to simplicial distributions.}

\begin{corollary}
If $\catsSet(X,D(NM))$ contains a contextual distribution then $\catsSet(X,D(\Dec^0 NM))$ also contains a contextual distribution.
\end{corollary}
\begin{proof}
Suppose that $p \in \catsSet(X,D(NM))$ is a contextual distribution. By Lemma \ref{TheSeccc} we have that the map 
$D(d_0)_\ast:\catsSet(X,D(\Dec^0 NM)) \to \catsSet(X,D(NM))$ is surjective. Therefore there exists $q \in \catsSet(X,D(\Dec^0 NM) )$ such that $D(d_0)_\ast(q)=p$. By Proposition \ref{lem:Exten2} we conclude that $q$ is contextual.
\end{proof}

\begin{corollary} 
{The} polytope $\catsSet(X,D(NM))$ is a subpolytope of $\catsSet(X,D(Dec^0(NM)))$, i.e., the map $T_*$ is injective and convex.  
\end{corollary}
\begin{proof}
This follows from Proposition \ref{pro:TConv}, Proposition \ref{pro:Pro 2.15},
  and the fact that 
$T$ is injective.
\end{proof}
\begin{example}{\rm
Let $X$ be a $1$-skeletal simplicial set with $n$ non-degenerate $1$-simplices. 
Then $\catsSet(X,D(N\zz_2))\cong [0,1]^n$ is a subpolytope of $\catsSet(X,D(Dec^0 N\zz_2 ))$.
}
\end{example}
%


{Let $G$ be a finite group.}
Given $p,q\in D(G)$, the convolution product is defined by
$$
p\ast q(g) =\sum_{{g_1\cdot g_2}=g} p(g_1)q(g_2)
$$
where the sum runs over pairs $(g_1,g_2)\in G^2$ such that $g_1\cdot g_2=g$.

\begin{lem}\label{THommm} 
For distributions $p,q\in D(G^n)$, we have 
$T_n(p\ast q)=T_n(p)\ast T_n(q)$.
\end{lem}
\begin{proof}
Given $(g_1,\cdots,g_{n+1}) \in G^{n+1}$, we have
{
$$
\begin{aligned}
T_n(p \ast q)(g_1,\cdots,g_{n+1})
&=\frac{(p\ast q)(g_2,\cdots, g_{n+1})}{|G|} \\
&=\sum_{ x_i \cdot y_i=g_i\, \text{for} \,\, i \geq 2}\frac{p(x_2,\cdots, x_{n+1})q(y_2,\cdots, y_{n+1})}{|G|} \\
&=\sum_{ x_1 \cdot y_1=g_1}\frac{1}{|G|}
\sum_{ x_i \cdot y_i=g_i \, \text{for} \,\, i \geq 2}\frac{p(x_2,\cdots,x_{n+1})q(y_2,\cdots,y_{n+1})
}{|G|} 
\\
&=\sum_{ x_i \cdot y_i=g_i\, \text{for} \,\, i \geq 1}
\frac{p(x_2,\cdots,x_{n+1})}{|G|} 
\frac{q(y_2,\cdots,y_{n+1})}{|G|}
\\
&=\sum_{ x_i \cdot y_i=g_i
\, \text{for} \,\, i \geq 1}
T_n(p)(x_1,\cdots,x_{n+1})
T_n(q)(y_1,\cdots,y_{n+1})\\
&=(T_n(p) \ast T_n(q))(g_1,\cdots,g_{n+1}).
\end{aligned}
$$
}
%
%
%
\end{proof}
%

%
%

\begin{pro}
Let $G$ be a finite 
group.
Then the group $\catsSet(X,NG)$ can be embedded in the  monoid $
\catsSet(X,D(\Dec^0
NG))$ by the injective 
semi-group homomorphism
$T_\ast \circ (\delta_{NG})_\ast$.
Moreover, for every map $\varphi \in \catsSet(X,NG)$ that is not null-homotopic, 
the distribution $T_\ast \circ (\delta_{NG})_\ast
(\varphi)$ is strongly contextual.                 
\end{pro}
\begin{proof}
Follows directly from Lemmas \ref{TheSeccc}, \ref{THommm}, and Proposition \ref{NotnullSC}.
\end{proof}

Using this result, Proposition \ref{pro:number of not null-homotopic maps NG}, and part $(3)$ of Proposition \ref{pro:ElemConv} we obtain the following result. 
%
\begin{corollary}\label{cor:XNG}
Let $X$ be a $1$-skeletal {connected} simplicial set such that $n=|X_0|$ and $m$ denote the number of non-degenerate {$1$-simplices of $X$}. There exists at least $|G|^m-|G|^{n-1}$ {strongly} contextual vertices in
$\catsSet(X,D(\Dec^0
NG))$.
\end{corollary}

\subsection{Description of the faces}
\label{sec:description of the faces}

In this section we describe the face $\Facc(\varphi)$ defined at a deterministic distribution $\delta^\varphi$ (Definition \ref{FacdefNH}). We begin with an extension result satisfied by $D_R(NG)$.  
{Let us define
\begin{equation}\label{eq:q-delta}
(q,\delta^g):\Lambda_i[2]\to D_R(NG)
\end{equation}
by sending $d_0\sigma\mapsto\delta^g$ and the other non-degenerate face of $\sigma$ to $q$, where $\sigma=\sigma^{012}$ is the generating simplex of $\Delta[2]$. 
} 
 
\begin{lemma}\label{ComppTrain1}
Given $g \in G$, $q \in D_R(G)$, and {$i\in\set{1,2}$}. The map $(q,\delta^g)$ {in (\ref{eq:q-delta})} extends to the $2$-simplex: 
%
\begin{equation}\label{dia:Exten}
\begin{tikzcd}[column sep=huge,row sep=large]
\Lambda_i[2] \arrow[r,"(q{,}\delta^g)"] \arrow[d,hook] & D_R(NG) \\
\Delta[2] \arrow[ru,dashed,"p"'] &
\end{tikzcd}
\end{equation}
such that $$
D_R({d_i})(p)=
\begin{cases}
q\ast \delta^g  & \text{if}  \;\; {i=1}  \\ 
q\ast \delta^{g^{-1}}  & \text{if}  \;\; {i=2}.
\end{cases}
$$
Moreover, the extension is unique.
\end{lemma}
\begin{proof}
We will prove this in the case that $i=1$. Let us define {$p\in D_R(G\times G)$} 
by
$$
p(g_1,g_2)=
\begin{cases}
q(g_1)  & \text{if} ~ g_2 = g,\\
0 & \text{otherwise.}
\end{cases}
$$
%
%
%
For $x\in G$, we have
$$
\begin{aligned}
D_R(d_2)(p)(x) &=
\displaystyle\sum_{g_2 \in G}p(x,g_2)\\
&=p(x,g)=q(x),
\end{aligned}
$$ 
hence $D_R(d_2)(p)=q$. Moreover, if $x \neq g$, then
$$
D_R(d_0)(p)(x)=
\displaystyle\sum_{g_1 \in G}p(g_1,x)=0.
$$ 
Therefore $D_R(d_0)(p)=\delta^g$. 
Now, we prove that $D_R(d_1)({p})=q \ast \delta^g$:
$$
\begin{aligned}
D_R(d_1)(p)(x) &=\displaystyle\sum_{x_1\cdot x_2 =x}p(x_1,x_2)\\
&=
\displaystyle\sum_{x_2 \in G}p(x\cdot x_2^{-1},x_2) \\
&=p(xg^{-1},g)\\
&=q(xg^{-1})=(q \ast \delta^g)(x).
\end{aligned}
$$
Suppose there exists another extension $\tilde{p}$ in Diagram (\ref{dia:Exten}), 
then $\tilde{p}(g_1,g_2)=0$ for $g_2\neq g$ because 
$D_R(d_0)(\tilde{p})=\delta^g$.
{Therefore} we obtain   
$$\tilde{p}(g_1,g)=\displaystyle\sum_{g_2 \in G}\tilde{p}(g_1,g_2)
=d_2(\tilde{p})(g_1)=q(g_1).$$
\end{proof}
%


\begin{pro}\label{FacNull}
Let $X$ be a connected 
$1$-skeletal simplicial set.
Given a simplicial set map $\varphi: X \to NG$ and  $v \in X_0$, 
the map  
$$
\eta : \Facc(\varphi) \to D_R(G)
$$
defined by
$
\eta(p)=p_0(v)
$
is an injective $R$-convex map.
\end{pro}
\begin{proof}
We 
{begin}
by proving the injectivity of $\eta$. Given $p\in \Facc(\varphi)$,   
let $p':\Delta^0 \ast X \to D(NG)$ be the adjoint of $p$ under the adjunction in (\ref{eq:Cone vs declage adjunction}). 
Let $\sigma$ be an edge in $X$ such that $d_{0}(\sigma)$ or $d_1(\sigma)$ equal to $v$. By Lemma \ref{lem:transcomm} and  because $p \in \Facc(\varphi)$ we obtain 
$$
p'_1(\sigma)=(D_R(d_0) \circ p)_1(\sigma)=
\delta^{\varphi_1(\sigma)}.
$$ 
By Lemma \ref{ComppTrain1} 
$p_1(\sigma)=p'_2(c_0,\sigma)$ is determined by $p'_1(c_0,v)=p_0(v)$.  
Because $X$ is a $1$-skeletal connected simplicial set we obtain that $p$ is uniquely determined by $p_0(v)$.
For the convexity of $\eta$, given $Q \in D_R(\Facc(\varphi))$, we need to prove that 
$\eta(\nu(Q))=\mu_G\left(D_R(\eta)(Q)\right)$:
$$
\begin{aligned}
\eta(\nu(Q))&=
\nu(Q)_0(v)\\
&=
\sum_{p\in \Facc(\varphi)}
Q(p)p_0(v)
\\
&= \sum_{p\in \Facc(\varphi)}
Q(p)\eta(p) \\
&=\sum_{q\in D_R(G)}(\sum_{\eta(p)=q}
Q(p))q \\
&=
\sum_{q\in D_R(G)}
D_R(\eta)(Q)(q)q \\
&=\mu_G\left(D_R(\eta)(Q)\right).
\end{aligned}
$$

\end{proof}

\begin{lem}\label{lem:DistOnLine}
Let $\varphi: X \to NG$ be a simplicial set map  and $L \subset X$ be a line such that {$u$ and $v$ are the initial and the terminal vertices of $L$ (see Definition \ref{def:Llinee}).}
Then for every simplicial distribution $p \in \Facc(\varphi)$ we have 
$$
p_0(v)=p_0(u) \ast \delta^{(L,\varphi|_{L})}
$$
(For the definition of $(L,\varphi|_{L})$ see Equation (\ref{eq:C varphi}).)

\end{lem}
\begin{proof}
 Follows from Lemma \ref{ComppTrain1}. 
\end{proof}
\begin{defn}\label{Hvarph}{\rm
Let $X$ be a $1$-skeletal simplicial set  and $\varphi:X\to NG$ be a simplicial set map. 
We define $H_{\varphi}\subset G$ to be the subgroup  
generated by the elements $(C,\varphi|_{C})\in G$  for every circle $C\subset X$. 
We write $G/\varphi$ for the set of 
orbits in $G$ under the action of $H_{\varphi}$ by multiplication from the right.
}     
\end{defn}
%


{
We will identify $D_R(G/\varphi)$ with the set $D_R(G)^{H_\varphi}$ of distributions $p:G\to R$ satisfying $p(g\cdot h)=p(g)$ for all $h\in H_\varphi$ by regarding a distribution $q:G/\varphi\to R$ as a distribution $\tilde q:G\to R$ invariant under $H_\varphi$ by defining 
$$
\tilde q(g)= \frac{q(\bar g)}{|H_\varphi|}
$$
where $\bar g$ is the image of $g$ under the quotient map $G\to G/\varphi$.}

\begin{pro}\label{pro:FaccCharac}
Let $G$ be an abelian group, and $X$ be a connected $1$-skeletal simplicial set.
For $\varphi \in \catsSet(X,NG)$, the $R$-convex set $\Facc(\varphi)$ is isomorphic to $D_R(G/\varphi)$. 
\end{pro}
\begin{proof}
Fix $v \in X_0$. In Proposition \ref{FacNull} we showed that 
$$
\eta : \Facc(\varphi) \to D_R(G)
$$
defined by
$
\eta(p)=p_0(v)
$
is an injective convex map. It remains to prove that the image of $\eta$ is equal to $D_R(G/\varphi)$. 
{Let $C\subset X$ be a circle.}
{Since}
$X$ is connected there {exists} a line $L$ from any vertex $u$ in $C$ to the vertex $v$. By Lemma \ref{lem:DistOnLine} we have 
$$
p_0(v)=p_0(u) \ast \delta^{(L,\varphi|_{L})} \;\; \text{ and } \;\; 
p_0(u)=p_0(u) \ast \delta^{(C,\varphi|_{C})}.
$$
We conclude that 
$$
\begin{aligned}
p_0(v)&=p_0(u) \ast \delta^{(L,\varphi|_{L})}\\
&=
p_0(u) \ast \delta^{(C,\varphi|_{C})} \ast \delta^{(L,\varphi|_{L})}\\
&=
p_0(u) \ast \delta^{(L,\varphi|_{L})} \ast \delta^{(C,\varphi|_{C})}\\
&=
p_0(v) \ast \delta^{(C,\varphi|_{C})}.
\end{aligned}
$$
This implies that for every $g \in G$ we have 
$
p_0(v)(g)=p_0(v)(g\cdot(C,\varphi|_{C})^{-1})
$.
So for every ${h} \in H_{\varphi}$ we have 
$$
p_0(v)(g)=p_0(v)(g\cdot {h}).
$$
Therefore $p_0(v)$ is indeed a distribution on $G/\varphi$. 
On the other hand, 
for 
a distribution $q \in D_R(G/\varphi)$, {regarded as a distribution $\tilde q \in D_R(G)^{H_\varphi}$,}  
we can define $p \in \Facc(\varphi)$ by setting $p_0(v)={\tilde q}$.
{Then b}y Lemma \ref{ComppTrain1} and because $X$ is a connected $1$-skeletal simplicial set, this defines $p_1(\sigma)$ for every $\sigma \in X_1$.
{Moreover, this assignment is consistent since $p_0(w)=p_0(w)(g\cdot (C,\varphi|_{C}))$ for every $w \in X_0$, $g \in G$, and circle $C\subset X$.}
So we {obtain} a well-defined $p \in \Facc(\varphi)$ that satisfies $\eta(p)=q$.
\end{proof}

%
%
\begin{corollary}\label{cor:NullHomFaccc}
Let $G$ be an abelian group and $X$ be a connected $1$-skeletal simplicial set. The map $\varphi\in \catsSet(X,NG)$ is null-homotopic if and only if $\Facc(\varphi)\cong D_R(G)$. 
\end{corollary}
\begin{proof}
Follows from Corollary \ref{cor:NullHomAct} and Proposition \ref{pro:FaccCharac}.
\end{proof}

\begin{pro}\label{pro:FacUnique}
Let $G=\zz_d$ where $d$ is a prime number.
For every  $\varphi \in \catsSet(X,N\zz_d)$ that is not null-homotopic, the set $\Facc(\varphi)$ contains a unique 
simplicial distribution that is a strongly  contextual vertex
of $\catsSet(X,D_R(Dec^0N \zz_d))$. 
\end{pro}
\begin{proof}
By Corollary \ref{cor:NullHomAct} we have $(C,\varphi)\neq 0$ for some circle $C$ in $X$. Since the action of $H_{\varphi}$ on $\zz_d$ is transitive, the face $\Facc(\varphi)$ includes only a single distribution. 
By Proposition \ref{NotnullSC} and part (1) of Proposition \ref{pro:FaccPrimeFill} this distribution is a strongly contextual vertex. 
\end{proof}

\begin{defn}\label{def:PR}
{\rm
Consider a circle $C$ with $n$ non-degenerate $1$-simplices $\sigma_1,\cdots,\sigma_n$. The Popescu--Rohrlich (PR) box \cite{pr94} (see Figure {(\ref{fig:triangle-PR})}) is the simplicial  distribution 
$p \in \catsSet(C,D(
Dec^0 N\zz_2))$ such that  $p_{{\sigma}_i} \in \{p_+,p_-\}$  
 with the further restriction that the number of $p_-$'s is
odd, where 
\begin{equation}\label{eq:PR}
p_+^{ab}=\left\lbrace
\begin{array}{cc}
1/2 & b=0\\
0 & b =1\\
\end{array}
\right.
\;\;\;\;
p_-^{ab}=\left\lbrace
\begin{array}{cc}
0 & b =0\\
1/2 & b =1.\\
\end{array}
\right.
\end{equation} 
}
\end{defn}
\begin{cor}\label{ex:PRBox}
{\rm 
Each PR box gives a strongly   contextual vertex in $\catsSet(C, D(\Dec^0N\zz_2))$.
}
\end{cor}
\begin{proof}
A simplicial set map $\varphi:C \to N\zz_2$ is not null-homotopic if and only if $\sum_{i=1}^n\varphi_{\sigma_i}$ is an odd number (see Proposition \ref{pro:NullHomCirc}). Therefore for a PR box $p$, the distribution $D(d_0)\circ p$ is a deterministic distribution $\delta^{\varphi}$ such that $\varphi$ is not null-homotopic. 
By Proposition \ref{pro:FacUnique} we obtain that $p$ is a strongly contextual vertex in $(C,\Dec^0(N\zz_2))$.
\end{proof}

\begin{pro}
{Let $G$ be an abelian group and $X$ be a connected $1$-skeletal simplicial set, and let $R$ be an integral semiring.} The subset of deterministic distributions in $\catsSet(X,D_R(\Dec^0 NG))$ is given by  $\displaystyle{\bigsqcup_{\varphi \sim \ast}V(\Facc(\varphi))}$. 
\end{pro}
\begin{proof}
The following diagram shows that every deterministic distribution in $\catsSet(X,{D_R(\Dec^0 NG)})$ lies in $\Facc(\varphi)$ for some null-homotopic map $\varphi$:
\begin{equation}
\begin{tikzcd}[column sep=huge,row sep=large]
\catsSet(X,\Dec^0(NG)) 
\arrow[rr,hook]
 \arrow[d,"(d_0)_\ast"] && \catsSet(X,D_R(\Dec^0 NG))  
 \arrow[d,"D_R(d_0)_\ast"] \\
\catsSet(X,NG) 
 \arrow[rr,hook] && \catsSet(X,D_R(NG)) 
\end{tikzcd}
\end{equation}
By Proposition \ref{pro:Pro 5.14}  each deterministic distribution on $(X,\Dec^0(NG))$ is a  vertex in $\catsSet(X,D_R
(\Dec^0NG))$. So far we have proved that the set of deterministic distributions in $\catsSet(X,D_R(\Dec^0NG))$ is included in $\displaystyle{\bigsqcup_{\varphi \sim \ast}V(\Facc(\varphi))}$. The number of deterministic distributions in $\catsSet(X,D_R(\Dec^0NG))$ is equal to $|G|^{|X_0|-1}\cdot |G|=|G|^{|X_0|}$. On the other hand, by Lemma \ref{lem:Kerofd0} and Corollary \ref{cor:NullHomFaccc} we {obtain} that $|\displaystyle{\bigsqcup_{\varphi \sim \ast}V(\Facc(\varphi))}|=|G|^{|X_0|}$. 
\end{proof}

\subsection{{Compositories and lifting properties}}


It is well-known that $NG$ and $\Dec^0(NG)$ are both fibrant simplicial sets (Kan complex). 
{We begin by observing} that {$D_R(NG)$} and {$D_R(\Dec^0(NG))$} are not fibrant.
This makes the homotopical study of simplicial distributions, i.e., simplicial set maps of the form $X\to D_R(Y)$ where $Y=NG$ or $\Dec^0(NG)$, harder. 
{However, we show that these simplicial sets are compositories in the sense of \cite{flori2013compositories}, hence satisfy lifting with respect to certain diagrams reflecting a kind of higher composition.}

We begin with the simplicial set $\Dec^0(NG)$. 
By Proposition \ref{pro:Pro 2.12} and Corollary \ref{cor:XNG} there is a contextual simplicial distribution $p:\partial \Delta[2] \to D(\Dec^0(NG))$. Consider a commutative diagram
$$
\begin{tikzcd}[column sep=huge,row sep=large]
\partial\Delta[2] \arrow[d,hook] \arrow[r,"p"] & D(\Dec^0(NG)) \arrow[d,"\simeq"] \\
\Delta[2] \arrow[r] \arrow[ru,dashed] & \ast
\end{tikzcd}
$$
where the right vertical arrow is a weak equivalence {(induces isomorphism on all homotopy groups)} by Proposition \ref{pro: DNGContr}.
If we assume that  $D(\Dec^0(NG))$ is fibrant then the diagonal dashed arrow exists making the diagram commute.
{Then Example \ref{ex:DeltaClass} and Proposition \ref{lem:Exten} implies that $p$ is non-contextual, which is a contradiction.}
For $D(NG)$ we observe that 
$$ 
\catsSet(\Lambda_{0}[3], D(NG))=\catsSet(\Delta[0]\ast \partial \Delta[2], D(NG))\cong \catsSet(\partial \Delta[2], D(\Dec^0(NG)))
$$ 
Let us write $p':\Lambda_0[3] \to D(NG)$ for the transpose of the contextual simplicial distribution $p:\partial \Delta[2] \to D(\Dec^0(NG))$. 
Then according to Proposition \ref{pro:Pro 2.23}
the distribution $p'$ is also contextual {and} by {Proposition} \ref{lem:Exten} the  distribution $p'$ does not extend to $\Delta[3]$. {Therefore} $D(NG)$ is not fibrant.

{Next, we observe that simplicial distributions on the nerve of a category is a compository. This observation will also be useful in Section \ref{sec:logical categories} when we introduce the category associated to a simplicial distribution.
We begin by recalling some basic definitions from \cite{flori2013compositories} about compositories and gleaves. We introduce two ordinal maps.
\begin{itemize}
\item The {\it $k$-source map} is given by the canonical inclusion
$$
S^k:[k] \to [n] 
$$
sending $i\mapsto i$.
\item The {\it $k$-target map} 
$$
T^k:[k]\to [m]
$$
is defined by $T^k(i)=i+m-k$. 
\end{itemize}
Let $X$ be a simplicial set.  A pair of simplices $(x,y)\in X_m\times X_n$ is $k$-composable if
$$
T_k(x)=S_k(y)
$$
where $T_k=(T^k)^*$ and $S_k=(S^k)^*$ are the induced maps on the simplices of $X$. A {\it compository} is a simplicial set $X$ equipped with a composition 
$$
x\circ_k y \in X_{m+n-k}
$$
for every $k$-composable pair $(x,y)\in X_m\times X_n$ and every $k\in \NN$ satisfying certain axioms \cite[Definition 2.2.2]{flori2013compositories}. 
{Our} canonical example is the nerve of a category.

{Compositories coincide with specific kinds of gleaves on the simplex category. Let $\catC$ be a category equipped with a system of bicoverings \cite[Defintion 4.2.1]{flori2013compositories} and $\catD$ be a category with pullbacks.}
A gleaf on a category $\catC$ with values in another category $\catD$ is a pair $(\Gamma,g)$ consisting of a functor $\Gamma:\catC^\op\to \catD$ together with a gluing operation
$$
g:\Gamma(A)\times_{\Gamma(A\times_C B)}\Gamma(B) \to \Gamma(C)
$$
for every bicovering}
$$
\begin{tikzcd}
 & B\arrow[d] \\
A \arrow[r] & C
\end{tikzcd}
$$ 
{The gluing operations satisfy certain axioms \cite[Definition 4.3.1]{flori2013compositories}.}

{Next, we introduce our gluing operation.
For a semifield $R$ and set maps $f:X\to Z$, $g:Y\to Z$,
the canonical map
$$
D_R(X\times_Z Y)\to D_R(X)\times_{D_R(Z)}D_R(Y)
$$
has a section $T:D_R(X)\times_{D_R(Z)}D_R(Y)\to D_R(X\times_Z Y)$ defined by} 
\begin{equation}\label{Eq:Tsection}
T(P,Q)(x,y)=\begin{cases}
\frac{P(x)Q(y)}{D_R(f)(P)(f(x))}  & \text{if} \;\;   D_R(f)(P)(f(x))\neq 0 \\
0 & \text{otherwise}
\end{cases}
\end{equation}
{where $(P,Q)\in D_R(X)\times_{D_R(Z)}D_R(Y)$ and 
$(x,y)\in X\times_Z Y$;
see \cite[Lemma 4.10]{kharoof2022simplicial}.
}
%
%
Now we prove some properties of that section.
\begin{lem}\label{pro:SectionProp}
Consider the following diagram in $\catSet$:
\begin{equation}\label{Eq:propofGlu}
\begin{tikzcd}[column sep=huge,row sep=large]
& & Y 
\arrow[d,""]
 \\
W  \arrow[r,"f"'] & X  \arrow[r,""'] & Z
\end{tikzcd}
\end{equation}
Given $(P,Q)\in D_R(W)\times_{D_R(Z)}D_R(Y)$ we have the following:
\begin{enumerate}
    \item $T\left(P,T(D_R(f)(P),Q)\right)=T(P,Q)$
    \item $D_R(r)(T(P,Q))=T\left(D_R(f)(P),Q\right)$
    \item $T\left(P,D_R(r)(T(P,Q))\right)=T(P,Q)$
\end{enumerate}
where $r$ is the  map 
{
$$
r: W\times_Z Y \to X\times_ZY
$$
induced by $f$ in Diagram (\ref{Eq:propofGlu}).}
\end{lem}
\begin{proof}
First, we complete Diagram (\ref{Eq:propofGlu}) by adding the pullbacks 
\begin{equation}\label{Eq:Two pullbacks}
\begin{tikzcd}[column sep=huge,row sep=large]
W\times_X (X\times_Z Y) \arrow[d,""] \arrow[r,"r"]& 
X\times_Z Y \arrow[r,""] \arrow[d,"r_1"]& Y 
\arrow[d,"r_2"]
 \\
W  \arrow[r,"f"'] & X  \arrow[r,""'] & Z
\end{tikzcd}
\end{equation}
{Part (1): For} $(w,(x,y))\in W\times_X (X\times_Z Y)$, we have 
$$
\begin{aligned}
T\left(P,T(D_R(f)(P),Q)\right)(w,(x,y))&=\frac{P(w)\cdot T(D_R(f)(P),Q)(x,y)}{D_R(r_1)(T(D_R(f)(P),Q))(x)}\\ 
&=
\frac{P(w)\cdot T(D_R(f)(P),Q)(x,y)}{D_R(f)(P)(x)} \\
&= \frac{P(w)\cdot \frac{D_R(f)(P)(x)Q(y)}{D_R(r_2)(Q)(r_2(y))}}{D_R(f)(P)(x)} \\
& =\frac{P(w)Q(y)}{D_R(r_2)(Q)(r_2(y))}  \\
& =T(P,Q)(w,y).
\end{aligned}
$$
{Part (2): For} $(x,y)\in X \times_{Z} Y$, we have 
$$
\begin{aligned}
D_R(r)(T(P,Q))(x,y)&=\sum_{w\in W \; : \; f(w)=x}
T(P,Q)(w,(x,y))=\\
& =\sum_{w\in W \; : \; f(w)=x}\frac{P(w)Q(y)}{D_R(r_2)(Q)(r_2(y))} \\
& =\frac{\left(\sum_{w\in W \; : \; f(w)=x}P(w) \right) Q(y)}{D_R(r_2)(Q)(r_2(y))} \\
& =\frac{D_R(f)(P)(x) Q(y)}{D_R(r_2)(Q)(r_2(y))} \\
& =T\left(D_R(f)(P),Q\right)(x,y).
\end{aligned}
$$
Part (3) follows directly from the first two parts.
\end{proof}

\begin{lemma}\label{lem:TProperty}
{Consider} the following commutative diagram in $\catSet$:
$$
\begin{tikzcd}[column sep=huge,row sep=large]
& Y  \arrow[r,"\beta"] 
\arrow[d,"g"'] & Y' \arrow[dl,"{g'}"]
 \\
X  \arrow[r,"f"] \arrow[d,"\alpha"]
& Z  \\
X' \arrow[ru,"{f'}"']
\end{tikzcd}
$$
{Then we have} the following commutative diagram:
$$
\begin{tikzcd}[column sep=huge,row sep=large]
D_R(X)\times_{D_R(Z)}D_R(Y) \arrow[r,"T"] \arrow[d,"D_R(\alpha)\times D_R(\beta)"']
& D_R(X\times_Z Y) \arrow[d,"D_R(\alpha \times \beta)"]
 \\
D_R({X'})\times_{D_R({Z})}D_R({Y'})    \arrow[r,"T"'] & 
D_R({X'}\times_{Z} {Y'})
\end{tikzcd}
$$
\end{lemma}
\begin{proof}
Given  $(P,Q)\in D_R({X})\times_{D_R({Z})}D_R({Y})$ and $(x',y')\in {X'}\times_{Z} {Y'}$, we have 
\begin{equation}\label{eq:DRalphabeta}
\begin{aligned}
D_R(\alpha \times \beta)(T(P,Q))(x',y')&=
\sum_{(x,y)\in {X}\times_{Z} {Y} \,\,:\,\,
\alpha(x)=x',\, \beta(y)=y'} T(P,Q)(x,y)\\
&=\sum_{(x,y)\in {X}\times {Y} \,\,:\,\,
\alpha(x)=x',\, \beta(y)=y'} \frac{P(x)Q(y)}{D_R(f)(P)(f(x))}\\
&=\sum_{(x,y)\in {X}\times {Y} \,\,:\,\,
\alpha(x)=x',\, \beta(y)=y'} \frac{P(x)Q(y)}{D_R(f)(P)(f'(x'))}.
\end{aligned}
\end{equation}
In Equation (\ref{eq:DRalphabeta}) we used {the fact that} $\alpha(x)=x'$ {and} $\beta(y)=y'$ implies that 
$(x,y)\in X {\times_Z} Y$. We also have
$$
\begin{aligned}
T\left((D_R(\alpha)\times D_R(\beta))
(P,Q)\right)(x',y')&=
T\left(D_R(\alpha)(P),D_R(\beta)(Q)\right)(x',y') \\
&=\frac{D_R(\alpha)(P)(x')D_R(\beta)(Q)(y')}
{D_R(f')(D_R(\alpha)(P))(f'(x'))} \\
&=\frac{\sum_{x\in X \,\,:\,\,
\alpha(x)=x'}P(x) \sum_{ y\in Y \,\,:\,\,
\beta(y)=y'} Q(y) }
{D_R(f'\circ \alpha)(P)(f'(x'))} \\
&=
\sum_{(x,y)\in {X}\times {Y} \,\,:\,\,
\alpha(x)=x',\, \beta(y)=y'} \frac{P(x)Q(y)}{D_R(f)(P)(f'(x'))}.
\end{aligned}
$$

\end{proof}

\begin{pro}\label{pro:NCCompo}
Given a {small} category $\catC$ and  a semifield $R$, the simplicial set $D_R(N\catC)$ is a compository. 
\end{pro}
\begin{proof} 
Compositories 
are precisely the gleaves {on the simplex category} with bicoverings 
$$
\begin{tikzcd}[column sep=huge,row sep=large]
 & {[n]} \arrow[d,"T^n"] \\
{[m]} \arrow[r,"S^m"] & {[j]}
\end{tikzcd}
$$
where $n+m\geq j$; see \cite[Theorem 4.3.8.]{flori2013compositories}. 
%
First we explain {the} gluing operator. For the following pullback square 
\begin{equation}\label{eq:knm}
\begin{tikzcd}[column sep=huge,row sep=large]
[k] \arrow[r,"S^k"] \arrow[d,"T^k"']
& {[n]} 
\arrow[d,"T^n"]
 \\
{[m]}  \arrow[r,"S^m"'] & {[m+n-k]}
\end{tikzcd}
\end{equation}
we have the map 
\begin{equation}\label{eq:GluOrComp}
T:D_R(N\catC)_m\times_{D_R(N\catC)_k}D_R(N\catC)_n\to D_R(N\catC)_{m+n-k}
\end{equation}
{where w}e used  $D_R((N\catC)_m)=D_R(N\catC)_m$ and $(N\catC)_m\times_{(N\catC)_k} (N\catC)_n=(N\catC)_{m+n-k}$. {To show that $D_R(N\catC)$ is a gleaf we need to verify the three axioms: (a) identity axiom, (b) back-and-forth axiom, and (c) partial naturality axiom; {see \cite[Definition 4.3.1]{flori2013compositories}}.}

{(a) Identity axiom:} Given the pullback square 
$$
\begin{tikzcd}[column sep=huge,row sep=large]
{[n]} \arrow[d,"T^n"] \arrow[r,equal]& {[n]} 
\arrow[d,"T^n"]
 \\
{[j]}  \arrow[r,equal] & {[j]}
\end{tikzcd}
$$
we have the following commutative diagram 
$$
\begin{tikzcd}[column sep=large, row sep =large]
D_R(N\catC)_j \arrow[drr,bend left=15,""] \arrow[rd,"l"] 
\arrow[rdd,bend right=15,equal] &\\ 
&D_R(N\catC)_j\times_{D_R(N\catC)_n}D_R(N\catC)_n
\arrow[r,"\pi_2"]
 \arrow[d,"\pi_1"'] & D_R(N\catC)_n
 \arrow[d,""] \\
&D_R(N\catC)_j  
 \arrow[r] & D_R(N\catC)_n 
\end{tikzcd}
$$
The map $T:D_R(N\catC)_j\times_{D_R(N\catC)_n}D_R(N\catC)_n \to D_R(N\catC)_j$ is a section for the map $l$. Therefore 
$$
T=(\pi_1\circ l)\circ T= \pi_1 \circ (l \circ T)=\pi_1.
$$
Similarly, one can prove {the second case of} {$T=\pi_2$.}

{(b) Back-and-forth axiom:} Given the following pullback squares:
\begin{equation}\label{diag:PullsquaresDelt}
\begin{tikzcd}[column sep=huge,row sep=large]
{[k']} \arrow[d,"T^{k'}"] \arrow[r,"S^{k'}"]& 
{[k]} \arrow[r,"S^{k}"] \arrow[d,"T^k"]& {[n]}
\arrow[d,"T^n"]
 \\
{[m']}  \arrow[r,"S^{m'}"'] & {[m]}  \arrow[r,"S^m"'] & {[m+n-k]}
\end{tikzcd}
\end{equation}
we have to prove that the following diagram commutes:
\begin{equation}\label{Eq:TDiag}
\begin{tikzcd}[column sep=huge,row sep=large]
D_R(N\catC)_{m'}\times_{D_R(N\catC)_{k'}}D_R(N\catC)_n \arrow[r,"T"] \arrow[d,"T"']
& D_R(N\catC)_{m+n-k}
 \\
D_R(N\catC)_{m+n-k}   \arrow[r,""'] & D_R(N\catC)_m\times_{D_R(N\catC)_k}D_R(N\catC)_n \arrow[u,"T"]
\end{tikzcd}
\end{equation}
Diagram (\ref{diag:PullsquaresDelt}) induces the following diagram of pullbacks:
$$
\begin{tikzcd}[column sep=huge,row sep=large]
(N\catC)_{m+n-k} \arrow[d,""] \arrow[r,""]& 
(N\catC)_m \arrow[r,""] \arrow[d,""]& (N\catC)_{m'} 
\arrow[d,""]
 \\
(N\catC)_n  \arrow[r] & (N\catC)_k  \arrow[r,""'] & (N\catC)_{k'}
\end{tikzcd}
$$
So by part (3) of Lemma \ref{pro:SectionProp} we obtain that Diagram (\ref{Eq:TDiag}) commutes.

{(c) Partial naturality axiom: Follows from} Lemma \ref{lem:TProperty}.
\end{proof}

%
%

%
%
\begin{defn}\label{def:DeltaU}
{\rm
For a set $U$ let $\Delta_U$ denote the simplicial set whose $n$-simplices are given by the set $U^{n+1}$ and 
	the simplicial structure maps are given by
	$$
	\begin{aligned}
		d_i(x_0,x_1,\cdots,x_n) &=( x_0,x_1,\cdots,x_{i-1},x_{i+1},\cdots,  x_n) \\
		s_j(x_0,x_1,\cdots,x_n) &=( x_0,x_1,\cdots,x_{j-1},x_j,x_{j},x_{j+1},\cdots,  x_n).
	\end{aligned}
	$$
}	
\end{defn}	
For a group $G$, the space $\Delta_G$ is an alternative description of the d\'ecalage  $\Dec^0(NG)$.
\begin{lemma}\label{lem:DecDelta}
There is an isomorphism of simplicial sets
$$
\Delta_{G} \xrightarrow{\cong} \Dec^0(NG)
$$
defined in degree $n$ by sending $(g_0,g_1,\cdots,g_{n})$ to the tuple $(g_0,g_0^{-1}g_1,\cdots,g_{n-1}^{-1}g_n)$.  
\end{lemma}
\begin{defn}\label{def:Czzd}
{For a set $U$} we define $\catC_{U}$ to be the category with objects $x \in U$ and a unique morphism, denoted by $(x,y):x\to y$, between any two objects. 
\end{defn}
Note that $\Delta_{U}=N \catC_{U}$, so by Proposition \ref{pro:NCCompo} we {obtain} the following:

\begin{corollary}\label{DRNGCompo}
Let $R$ be a semifield and let $G$ be a group. The simplicial sets $D_R(NG)$ and $D_R(\Delta_G)$ are compositories.  %
\end{corollary}
{
Note that the $k$-th order composition that $D_R(N\catC)$ is  equipped with given in (\ref{eq:GluOrComp})  implies that for $T^k: [k] \to [m]$ and $S^k: [k] \to [n]$, the following induced diagram admits a lifting for every simplicial distribution $p$:}
\begin{equation}\label{dia:OrderdLift}
\begin{tikzcd}[column sep=huge,row sep=large]
\Delta[m] \cup _{\Delta[k]} \Delta[n] \arrow[d,hook] \arrow[r,"p"] & D_R(N\catC) \\
\Delta[m+n-k]  \arrow[ru,dashed] 
\end{tikzcd}
\end{equation} 
For $D_R(NG)$ and $D_R(\Delta_{G})$ we will generalize this extension property using the fact that $NG$ and $\Delta_{G}$ are fibrant simplicial sets.

\begin{pro}\label{pro:Exten}
Let $Z$ be a simplicial subset of $X$.
\begin{enumerate}
\item Every non-contextual distribution
$p:Z \to D_R(\Delta_G)$  extends to $X$.
\item Assume that $Z \hookrightarrow X$ is a weak equivalence. Then
every non-contextual distribution $p:Z \to D_R(NG)$  extends to $X$.   
\end{enumerate}
\end{pro}
\begin{proof}
{By Proposition \ref{pro: DNGContr}} $\Delta_G \cong \Dec^0(NG)$ is a contractible fibrant space, 
so the induced map $\catsSet(X,\Delta_G) \to \catsSet(Z,\Delta_G )$ is surjective, which implies that $D_R(\catsSet(X,\Delta_G)) \to D_R(\catsSet(Z,\Delta_G))$ is also surjective. Then   {part (1) is obtained} by Diagram (\ref{diag:ImpDiag}) when $Y=\Delta_G$ and the map $f$ is the inclusion $Z \hookrightarrow X$.
The proof of part (2) is similar.
\end{proof}
\begin{cor} 
Let $R$ be a semifield, and  $Y$ denote $\Delta_G$ or $NG$. Given inclusions $[k] \hookrightarrow [m]$, $[k] \hookrightarrow [n]$, 
and  $\Delta[m] \cup _{\Delta[k]} \Delta[n] \hookrightarrow \Delta[j]$,
 the following diagram admits a lifting for every simplicial distribution $p$:
$$
\begin{tikzcd}[column sep=huge,row sep=large]
\Delta[m] \cup _{\Delta[k]} \Delta[n] \arrow[d,hook] \arrow[r,"p"] & D_R(Y)  \\
\Delta[j]  \arrow[ru,dashed] 
\end{tikzcd}
$$
\end{cor}
\begin{proof}
By \cite[Corollary 4.6]{okay2022simplicial} every simplicial distribution in $\catsSet(\Delta[m] \cup _{\Delta[k]} \Delta[n],D_R(Y))$ is non-contextual. 
{Therefore the result follows from}  Proposition \ref{pro:Exten}. 
\end{proof}

\section{Homotopical methods}
\label{sec:homotopical methods}
  
A basic question in the theory {of} simplicial distributions is to understand the role of homotopy in 
the study
of contextuality. In this section we bring in tools from homotopy theory to shed light on this question. 
  
\subsection{Collapsed   distributions}
\label{sec:collapsed distributions}



{The degeneracy map $s^i:[n]\to [n-1]$ induces a simplicial set map $s^i:\Delta[n]\to \Delta[n-1]$ 
defined
by $\sigma^{01\cdots n}\mapsto s_i(\sigma^{01\cdots(n-1)})$.}

\begin{defn}\label{def:CollMap}
{\rm
Let $X$ be a simplicial set and $x\in X_n$ be  an $n$-simplex. For $0\leq i \leq n-1$, we define $X^{s^i,x}$ by the following push-out diagram
%
%
\begin{equation}
\begin{tikzcd}[column sep =huge, row sep =large]
\Delta[n]
\arrow[r,"x"]
 \arrow[d,"s^i"'] & X  
 \arrow[d,"\pi^{s^i,x}"] \\
\Delta[n-1]
 \arrow[r] & X^{s^i,x} 
\end{tikzcd}
\end{equation}
The map $\pi^{s^i,x}$ will be called a {\it collapsing map}. When $s^i$ and $x$ {are understood from the context} 
we will write $\pi=\pi^{s^i,x}$ and $\bar X=X^{s^i,x}$.
}
\end{defn}
 
{\rm
{Note that} 
$\pi^{s^i,x}$ 
is surjective, {which} 
implies the injectivity of the induced map $(\pi^{x,s^i})^\ast: \catsSet(\bar{X},W) \to \catsSet(X,W)$ for every simplicial set $W$.
}
}
%

\begin{lem}\label{lem:collMap} 
A simplicial set map $\varphi:X\to W$ lies in the image $ (\pi^{x,s^i})^\ast = \pi^\ast: \catsSet(\bar{X},W) \to \catsSet(X,W)$ (factors through the collapsing map $\pi$) if and only if $\varphi_n(x) \in s_i(W_{n-1})$. 
\end{lem}
\begin{proof}
The map $\varphi \in \Image \pi^\ast $ if and only if there is $\psi:\bar X \to W$ such that $\psi \circ \pi=\varphi$. This happens if and only if there {exists} $w:\Delta[n-1] \to W$ 
{making the following diagram commute}
\begin{equation}\label{dia:Ypsi}
\begin{tikzcd}[column sep=large, row sep =large]
\Delta[n]
\arrow[r,"x"]
 \arrow[d,"s^i"] & X
 \arrow[ddr,bend left=35,"\varphi"]
 \arrow[d,"\pi"] \\
\Delta[n-1] \arrow[rrd,bend right=30,"w"'] 
 \arrow[r] & \bar X \arrow[rd,"\psi"]\\
&& {W}
\end{tikzcd}
\end{equation}
{This}
is equivalent to {saying that} there {exists} $w \in W_{n-1}$ such that 
$\varphi_n(x)=s_i^W(w)$.
%
%
%
\end{proof}

\begin{cor}\label{cor:CollaDostCha} 
A simplicial distribution $p: X\to D_R(Y)$ lies in the image $(\pi^{x,s^i})^\ast: \catsSet(\bar{X},D_R(Y)) \to \catsSet(X,D_R(Y))$  
 if and only if $p_n(x)(y)=0$ for every $y \in  Y_n-\Image s^Y_i$.
\end{cor}

\begin{thm}\label{thm:CollapThm}
Let $\pi=\pi^{x,s^i}:X \to \bar X$ be a collapsing map. For a simplicial distribution 
$$p : \bar X\to D_R(Y)$$ 
we have the following:
\begin{enumerate}
\item $p$ is contextual if and only if  $\pi^\ast (p)$ is contextual.
\item $p$ is strongly contextual if and only if $\pi^\ast (p)$ is strongly contextual.

\item If the semiring $R$ is also  integral, then $p$ is  a vertex if and only if $\pi^\ast (p)$ is a vertex.
\item $p$ is a deterministic distribution if and only if $\pi^\ast (p)$ is a deterministic distribution.
\item If the semiring $R$ is also  integral, then $p$ is a contextual vertex if and only if $\pi^\ast (p)$ is a contextual vertex. 
\end{enumerate}
\end{thm}
\begin{proof}
Part $(1)$: Lemma \ref{lem:Exten} implies that if $\pi^\ast (p)$ is contextual then $p$ is contextual. For the other direction, suppose that we have $d \in D_R(\catsSet(X,Y))$ such that $\Theta_{X,Y}(d)=\pi^\ast (p)$. Given 
$y \in Y_n-\Image{s_i^Y}$,  Corollary \ref{cor:CollaDostCha} implies that
$$
\Theta_{X,Y}(d)_n(x)(y)=\sum_{\varphi:\, \varphi_n(x)=y}d(\varphi)=(\pi^\ast (p))_n(x)(y)=0.
$$
Since $R$ is a zero-sum-free semiring we conclude that $d(\varphi)=0$ for every $\varphi$ with $\varphi_n(x) \notin \Image s_i^Y$.  Therefore according to Lemma \ref{lem:collMap} we obtain that $d(\varphi)=0$ for every $\varphi \notin \Image \pi^\ast$. We define $\tilde{d}\in D_R(\catsSet(\bar X,Y))$ to be
$$
\tilde{d}(\psi)=d(\pi^\ast(\psi)).
$$
Then $\tilde{d}$ is {a} well-defined distribution  since 
\begin{equation}\label{eq:Sum111}
\sum_{\psi \in \catsSet(\bar X,Y)} \tilde{d}(\psi)=\sum_{\psi \in \catsSet(\bar X,Y)}d(\pi^\ast(\psi))
=\sum_{\varphi \in \Image{\pi^\ast}} d(\varphi)=1.
\end{equation}
{In Equation (\ref{eq:Sum111}) we used the fact that $\pi^\ast$ is injective. We will use this fact again to {obtain}}
$$
D_R(\pi^\ast)(\tilde{d})(\varphi)=\sum_{\pi^\ast
(\psi)=\varphi}\tilde{d}(\psi)=\begin{cases}
d(\varphi)  &  \varphi \in \Image{\pi^\ast}  \\
0   &   \text{otherwise}.
\end{cases}
$$
Therefore $D_R(\pi^\ast)(\tilde{d})=d$. By Diagram (\ref{diag:ImpDiag}) when $f=\pi$, we have that 
$$
\pi^\ast(\Theta_{\bar X,Y}(\tilde{d}))=\Theta_{X,Y}(D_R(\pi^\ast)(\tilde{d}))=\Theta_{X,Y}(d)=\pi^\ast(p).
$$
By the injectivity of $\pi^\ast$ we obtain that $\Theta_{\bar X,Y}(\tilde{d})=p$.

Part (2): If $\pi^\ast (p)$ is strongly contextual, then $p$ is strongly contextual by  \cite[Lemma 5.19, part (1)]{kharoof2022simplicial}.
Now, suppose that $\varphi \in \supp(\pi^\ast(p))$. Then we have that $(\pi^\ast (p))_n(x)(\varphi_n(x))\neq 0$. 
By Corollary \ref{cor:CollaDostCha} we conclude that $\varphi_n(x) \in \Image s^Y_i$. 
By Lemma \ref{lem:collMap} we obtain that $\varphi=\psi \circ \pi$ for some $\psi \in \catsSet(\bar X,Y)$. Therefore, for every $\sigma \in X_m$ we have 
$$
p_m(\pi_m(\sigma))(\psi_m(\pi_m(\sigma))) \neq 0.
$$
Since $\pi$ is surjective we conclude that $\psi \in \supp(p)$.   

Part (3): {Suppose that $\pi^\ast(p)$ is a vertex.} According to Proposition \ref{pro:Pro 2.15} the map $\pi^\ast$ is $R$-convex map. So by part $(3)$ of Proposition \ref{pro:ElemConv} every vertex in the preimage of $\pi^\ast (p)$ under $\pi^\ast$
is a
vertex in $\catsSet(\bar X,D_R(Y))$. By the injectivity of $\pi^\ast$
this preimage contains only $p$, therefore $p$ is a vertex. For the converse, suppose that $p$ is a vertex in $\catsSet(\bar X,D_R(Y))$. Given $Q\in D_R(\catsSet(X,D_R(Y)))$ such that $\nu^{\catsSet(X,D_R(Y))}(Q)=\pi^\ast (p)$. 
Given 
$y \in Y_n-\Image s^Y_i$,  by corollary \ref{cor:CollaDostCha} and Equation (\ref{piForm}) 
we have the following:
$$
\sum_{q\in \catsSet(X,D_R(Y))} Q(q)q_n(x)(y)=(\pi^\ast(p))_n(x)(y)=0.
$$
Since $R$ is an integral, zero-sum-free semiring we have $q_n(x)(y)=0$ for every $q \in \catsSet(X,D_R(Y))$ with $Q(q)\neq 0$. Using Corollary \ref{cor:CollaDostCha} we conclude that every $q$ with $Q(q) \neq 0$ is in the image of $\pi^\ast$. So we can define $\tilde{Q} \in D_R(\catsSet( \bar X,D_R(Y)))$ by $
\tilde{Q}(q')=Q(\pi^\ast (q'))
$. Note that by the injectivity of $\pi^\ast$ we have
$$
D_R(\pi^\ast)(\tilde{Q})(q)=\sum_{\pi^\ast
(q')=q}\tilde{Q}(q')=\begin{cases}
Q(q)  &  q \in \Image{\pi^\ast}  \\
0   &   \text{otherwise.}
\end{cases}
$$
Therefore $D_R(\pi^\ast)(\tilde{Q})=Q$. Using the convexity of $\pi^\ast$ we obtain
$$
\pi^\ast\left(\nu^{\catsSet
(\bar X,D_R(Y))}(\tilde{Q})\right)=\nu^{\catsSet
(X,D_R(Y))}\left(D_R(\pi^\ast)(\tilde{Q})\right)=\nu^{\catsSet
(X,D_R(Y))}(Q)=\pi^\ast(p).
$$
Again, using the fact that $\pi^\ast$ is injective we {obtain} that 
$\nu^{\catsSet
(\bar X,D_R(Y))}(\tilde{Q})=p$. Because $p$ is a vertex we have  $\tilde{Q}=\delta^{p}$. Therefore 
$Q=D_R(\pi^\ast)(\tilde{Q})=D_R(\pi^\ast)(\delta^p)=\delta^{\pi^\ast (p)}$.

Part (4): It is easy to see that if $p$ is deterministic, then $\pi^\ast (p)$ is also deterministic. Now, suppose that $\pi^\ast (p)=\delta_Y \circ \varphi$ for some $\varphi \in \catsSet(X,Y)$. Then by Corollary \ref{cor:CollaDostCha} for every $y \in Y_n -\Image s_i^Y$ we have 
$$
\delta^{\varphi_n(x)}(y)=(\delta_Y\circ \varphi)_n(x)(y)=(\pi^\ast (p))_n(x)(y)=0.
$$
We conclude that $\varphi_n(x) \in \Image s_i^Y$. {Then} by Lemma \ref{lem:collMap} we {obtain} that $\varphi=\pi^\ast (\psi)$ for some $\psi \in \catsSet(\bar X,Y)$. Therefore we have 
$$
\pi^\ast (p)=\delta_Y \circ \varphi=\delta_Y \circ \pi^\ast (\psi) = \delta_Y \circ \psi \circ \pi = \pi^\ast(\delta_Y \circ \psi).
$$
The injectivity of $\pi^\ast$ implies that $p=\delta_Y \circ \psi$. 

Part $(5)$: Follows from parts (1) and (3).
\end{proof}
%

Theorem \ref{thm:CollapThm} generalizes  \cite[Theorem 3]{e25081127}. It is going to be used in the proof of {Proposition \ref{Pro:CharofLSC}.} For other applications of the collapsing method see \cite[Sections 5.1 and 5.2]{e25081127}.
%

\subsection{The associated category of a simplicial distribution}
\label{sec:logical categories}

For the rest of the paper we assume that 
$X$
is $1$-skeletal. 
 
%
%

%
Recall that a simplicial distribution $p:\Delta[1]\to D_R(\Delta_{\ZZ_d})$ is specified by a distribution $p_\sigma\in D_R(\ZZ_d^2)$ at the non-degenerate simplex $\sigma=\sigma^{01}$. For us it will be convenient, in this section and the next one, to regard this distribution as a $d\times d$-matrix $M$ where
$$
M_{a,b} = p_\sigma^{ab}.
$$
For convenience we index the rows and columns of such a matrix by $0,1,\cdots, d-1$.

\begin{rem}\label{rem:PRBox}
{\rm
A simplicial distribution $\Delta[1]\to D_R(\Delta_{\ZZ_d})$ can be regarded as a simplicial distribution $\Delta[1]\to D_R(\Dec^0(N\ZZ_d))$, and vice versa, using the isomorphism $\Delta_{\ZZ_d}\xrightarrow{\cong} \Dec^0(N\ZZ_d)$ of Lemma \ref{lem:DecDelta}, which  in degree $1$ is given by $(a,b)\mapsto (a,-a+b)$.   For example, the distributions $p_{+}$ and $p_{-}$ can be represented by the matrices
$
\begin{pmatrix} 
	\frac{1}{2} & 0 \\
	0 & \frac{1}{2}\\
	\end{pmatrix}$ and
$\begin{pmatrix} 
	0 & \frac{1}{2} \\
	\frac{1}{2} & 0 \\
	\end{pmatrix}
$, respectively. See Figure (\ref{fig:triangle-PR}) and Equation (\ref{eq:PR}).
}
\end{rem}

\begin{defn}\label{def:ComposDist}
{\rm 
{Let $R$ be a semifield, and} let 
$X=\Delta[1] \vee_{\Delta[0]} \Delta[1]$ with the non-degenerate $1$-simplices 
$\sigma,\tau \in X_1$ such that 
$$
d_1(\sigma)=x, \; 
d_0(\sigma)=d_1(\tau)=y, \; d_0(\tau)=z.
$$
For a simplicial distribution $p:X \to D_R(\Delta_{\zz_d})$ we define the composition 
$p_{\sigma}\circ p_{\tau}\in D_R(\zz_d\times \zz_d)$ to be 
$$
(p_{\sigma}\circ p_{\tau})^{ab}=
\begin{cases}
\sum_{c\in \zz_d}
\frac{p_{\sigma}^{ac}
p_{\tau}^{cb}}{p_{y}^c} & \text{if} \; p_{y}^{c}\neq 0 \\
0  & \text{if} \; p_{y}^{c}   = 0.
\end{cases}
$$ 
}
\end{defn}
\begin{rem}
{\rm 
{As we showed in Corollary \ref{DRNGCompo} the simplicial set $D_R(\Delta_{\zz_d})$ is a compository.} The pair $(p_{\sigma},p_{\tau})$ of Definition \ref{def:ComposDist} is $0$-composable, that is, $p_{\sigma}\circ_0 p_{\tau}$ exists. 
{O}ur composition $p_{\sigma}\circ p_{\tau}$ is actually equal to $D_R(d_1)(p_{\sigma}\circ_0 p_{\tau})$.
One can think about $p_{\sigma}\circ p_{\tau}$ as a simplicial distribution on an edge from $x$ to $z$ where by properties (a) and (d) of \cite[Definition 2.2.2]{flori2013compositories}  we have  
\begin{equation}\label{eq:RestDist}
(p_{\sigma}\circ p_{\tau})_x=p_x \; \text{and} \; (p_{\sigma}\circ p_{\tau})_z=p_z.
\end{equation}
See Figure (\ref{PicCopo}).
}
\end{rem}

\begin{figure}[h!]
\centering
\includegraphics[width=.28\linewidth]{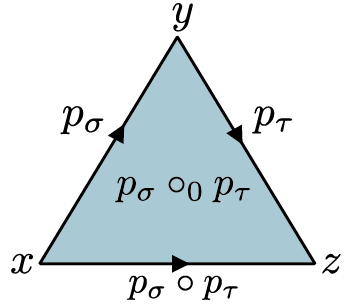} 
\caption{}
\label{PicCopo}
\end{figure} 
\begin{prop} 
{Let $R$ be a semifield} and let $X=\Delta[1] \vee_{\Delta[0]} \Delta[1]\vee_{\Delta[0]} \Delta[1]$ with the non-degenerate $1$-simplices 
$\sigma,\tau,\theta \in X_1$ 
satisfying
$$ 
d_0(\sigma)=d_1(\tau), \; d_0(\tau)=d_1(\theta).
$$
For a simplicial distribution $p:X \to D_R(\Delta_{\zz_d})$ we have
\begin{equation}\label{eq:Associa}
p_{\sigma} \circ (p_\tau \circ p_{\theta})=(p_{\sigma} \circ p_\tau) \circ p_{\theta}.  
\end{equation}
\end{prop}
\begin{proof}
By \cite[Proposition 2.2.8]{flori2013compositories} the composition $p_{\sigma} \circ_0 p_\tau \circ_0 p_{\theta}$ is a well-defined distribution in $D_R(\zz_d^4)$. One can see that both sides of Equation (\ref{eq:Associa}) are equal to $ D_R(d_1{\circ d_1})(p_{\sigma} \circ_0 p_\tau \circ_0 p_{\theta})$.
\end{proof} 
\begin{defn}\label{def:C(X,p)}
{\rm
Let $R$ be a  semifield and $X$ be a $1$-skeletal simplicial set. Consider a simplicial distribution $p:X\to D_R(\Delta_{\ZZ_d})$.
A non-degenerate $1$-simplex $\sigma \in X_1^\circ$ 
can be regarded
as an edge from $d_1(\sigma)$ to $d_0(\sigma)$. To $\sigma$ we also associate an edge pointing in the opposite direction denoted by
$\sigma^T$, i.e., 
an edge from $d_0(\sigma)$ to $d_1(\sigma)$.
Consider the directed graph $(X_0,E(X))$ where $E(X)=\{\sigma,\sigma^T\,|\, \sigma \in X_1^\circ\}$ and the associated free category $\catC(X)$ on this directed graph.
We define a new  
category $\catC(X,p)$ as follows:
\begin{itemize}
\item objects are given by the set $X_0$ of $0$-simplices,
\item morphisms $q:x\to y$ are given by  compositions $q=q_1\circ q_2\circ \cdots \circ q_k$ of distributions for  each morphism in the free category $\catC(X)$ represented as a composition of $1$-simplices
$$
x=x_0 \xrightarrow{\tau_1} x_1 \xrightarrow{\tau_2} \cdots \xrightarrow{\tau_k} x_k=y
$$
and the corresponding distributions $q_i$ are defined as follows: 
$$
q_i=
\begin{cases}
p_{\sigma} &  \text{if} \; \tau_i=\sigma   \\
p_{\sigma}^T & \text{if} \; \tau_i=\sigma^T .
\end{cases} 
$$
\end{itemize}
}
\end{defn} 
\begin{lem}\label{lem:Impo}
Let $X$ be a $1$-skeletal simplicial set and $p:X\to D_R(\Delta_{\ZZ_d})$ be a simplicial distribution. For every $q \in \catC(X,p)(x,y)$ we have 
$$
D_R(d_1)(q)=p_x \;\text{ and }\; D_R(d_0)(q)=p_y.
$$
\end{lem}
\begin{proof} Follows from Equation (\ref{eq:RestDist}) and the fact that $D_R(d_0)(p_{\sigma}^T)=D_R(d_1)(p_{\sigma})$.
\end{proof}

Now, we will focus {to} the case {where} the semiring is the Boolean algebra $\BB=\set{0,1}$. In this case the composition in Definition \ref{def:ComposDist} is exactly the product of matrices. Let $\Pi :\RR_{\geq 0}\to \BB$ denote the semiring homomorphism defined by
$$
\Pi(x) = \left\lbrace
\begin{array}{cc}
1 & x>0 \\
0 & \text{otherwise.}
\end{array}
\right.
$$
For simplicial sets $X,Y$ we will consider the induced map
\begin{equation}\label{eq:PiRealToBool}
\Pi_*: \catsSet(X,D(Y)) \to \catsSet(X,D_{\BB}(Y))
\end{equation}
The{n the} following diagram commutes:
\begin{equation}\label{dia:PiPi}
\begin{tikzcd}[column sep=huge,row sep = large]
D(\zz_d\times \zz_d) \times_{D(\zz_d)} D(\zz_d\times \zz_d) \arrow[r,"\circ"] \arrow[d,"\Pi_*\times \Pi_*"] & D(\zz_d\times \zz_d) \arrow[d,"\Pi_*"]\\
D_{\BB}(\zz_d\times \zz_d) \times_{D_{\BB}(\zz_d)} D_{\BB}(\zz_d\times \zz_d) \arrow[r,"m"] & D_{\BB}(\zz_d\times \zz_d)
\end{tikzcd}
\end{equation}
%
%
%
where $m$ denotes the matrix multiplication. Therefore we have a functor
$$
\Pi_*: \catC(X,p) \to \catC(X,\Pi_*(p))
$$
%

\begin{ex}\label{ex:ABCDU}
{\rm
Let us consider $\catC(X,p)$ for a simplicial distribution $p:X\to D_\BB(\Delta_{\ZZ_2})$.
 We will use the following notation: 
$$ 
A=\begin{pmatrix} 
	1 & 1  \\
 1 & 0 \\
	\end{pmatrix} 
,\;\;	B=\begin{pmatrix} 
	1 & 1  \\
 0 & 1 \\
	\end{pmatrix} 
,\;\; 
D=\begin{pmatrix} 
	0 & 1  \\
 1 & 1 \\
	\end{pmatrix} 
,\;\;	
U=\begin{pmatrix} 
	1 & 1  \\
 1 & 1 \\
	\end{pmatrix}. 		
$$
The set $G=\{A,B,B^T,D,U\}$ is a semi-group with  the following rules:
\begin{enumerate}
\item $UM=MU=U$, $\forall M \in G$,
\item $U^T=U$, $A^T=A$, $D^T=D$,

\item $AA=DD=U$, $BB=B$,

\item $MM^T=M^TM=U$, $\forall M \in G$,

\item $AD=B$, $DA=B^T$,

\item $AB=B^TA=U$, $BA=AB^T=A$, $BD=DB^T=U$, $DB=B^TD=D$. 
\end{enumerate} 


\begin{figure}[h!]
\centering
\begin{subfigure}{.5\textwidth}
  \centering
  \includegraphics[width=.5\linewidth]{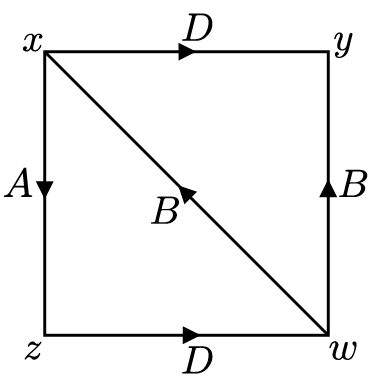}
  \caption{}
  \label{fig:CX}
\end{subfigure}%
\begin{subfigure}{.5\textwidth}
  \centering
  \includegraphics[width=.75\linewidth]{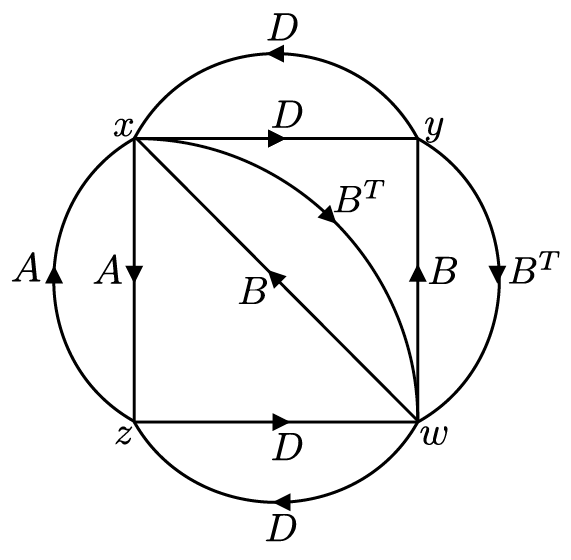}
  \caption{}
  \label{fig:CXp}
\end{subfigure}
\caption{
}
\label{fig:C}
\end{figure}

Let $p: X \to D_{\BB}(\Delta_{\zz_2})$ be the following simplicial distribution: 
$$
p_{\sigma_1}=A, \;\; p_{\sigma_2}=D, \;\;  p_{\sigma_3}=B, \;\;  p_{\sigma_4}=D, \;\;  p_{\sigma_5}=B   
$$
as described in Figure (\ref{fig:CX}).
We add the corresponding opposite 
edges and the transposed matrices for each distribution as in 
Figure (\ref{fig:CXp}).
%
By adding all the possible compositions we obtain
$$
\begin{aligned} 
&\catC(X,p)(x,y)=\{U,D,B\} &
&\catC(X,p)(x,z)=\{U,A,D\} \\
&\catC(X,p)(x,w)=\{U,B^T,B\} &
&\catC(X,p)(y,z)=\{U,D,B^T\} \\
&\catC(X,p)(y,w)=\{U,B^T,D\} &
&\catC(X,p)(z,w)=\{U,D,A\} \\
&\catC(X,p)(x,x)=\{I,U,B,B^T\} &
&\catC(X,p)(y,y)=\{I,U,D\}\\
&\catC(X,p)(z,z)=\{I,U,B,B^T\} &
&\catC(X,p)(w,w)=\{I,U,B,B^T\}.
\end{aligned}
$$
}
\end{ex}
%


\begin{defn}\label{def:logical category}
{\rm 
Let $d\geq 2$ be an integer, we will write $\Mat_d(\BB)$ to denote the monoid of $d\times d$ matrices under matrix multiplication. A small category $\catC$ is called a {\it logical category} if for every objects $x,y\in \Obj(\catC)$ the morphism set $\catC(x,y)$ is a subset of $\Mat_d(\BB)-\{0\}$
and the following diagram commutes
$$
\begin{tikzcd}[column sep=huge,row sep = large]
\catC(x,y) \times \catC(y,z) \arrow[r,"\circ"] \arrow[d,hook] & \catC(x,z) \arrow[d,hook]\\
\Mat_d(\BB) \times \Mat_d(\BB) \arrow[r,"m"] & \Mat_d(\BB)
\end{tikzcd}
$$ 
where $m$ denotes the matrix multiplication. A logical category is said to be {\it symmetric} if $M\in \catC(x,y)$ implies that the transposed matrix $M^T$ belongs to $\catC(y,x)$.
}
\end{defn}
For a simplicial distribution $p:X \to D_{\BB}(\Delta_{\zz_d})$ where $X$ is $1$-skeletal, the associated category $\catC(X,p)$ is a symmetric logical category. 
\begin{defn}\label{def:support of C(X,p)}
{\rm 
Let $F:\catC\to \catC_{\ZZ_d}$ be a functor, where $\catC$ is a logical category and {$\catC_{\zz_d}$ as in Definition \ref{def:Czzd}.} 
We say $F$ is in the {\it support} of $\catC$ if $M_{F(M)} =1 $ for every  $M\in \catC(x,y)$ and $x,y\in \Obj(\catC)$.
}
\end{defn}

%
%

%

Note that a functor $F:\catC \to \catC_{\zz_d}$ is determined by the values on the objects. 
Therefore there is a one-to-one correspondence between  functors $F:\catC(X,p) \to \catC_{\zz_d}$ and simplicial set maps $\varphi:X \to \Delta_{\zz_d}$. 

\begin{prop}\label{pr:DisSCCatSC}
Let $X$ be a $1$-skeletal simplicial set and $p:X\to D_\BB(\Delta_{\ZZ_d})$ be a simplicial distribution.
A simplicial  set map $\varphi:X\to \Delta_{\ZZ_d}$ is in the support of $p$ if and only if the corresponding functor $F_\varphi:\catC(X,p) \to \catC_{\zz_d}$ is in the support of $\catC(X,p)$.
\end{prop}
\begin{proof} For $\tau \in X_1^\circ$, let us denote $p_{\tau}$ by the matrix $M$. Recall that we have 
\begin{enumerate}
\item $M\in \catC(X,p)(d_1(\tau),d_0(\tau))
$,
\item $F_{\varphi}(M)=\varphi_{\tau}
$,
\item $M_{F_{\varphi}(M)}=
p_{\tau}(\varphi_{\tau})$.
\end{enumerate}
By (3) we conclude that if $F_{\varphi}$ is in the support of $\catC(X,p)$, then $\varphi \in \supp(p)$.

Now, suppose that $\varphi \in \supp(p)$.
Then $M_{F_{\varphi}(M)}=1$, and since
if 
$F_{\varphi}(M)=(a,b)$ then $F_{\varphi}(M^T)=(b,a)$, we obtain that $M^T_{F_{\varphi}(M^T)}=1$. Finally, using the fact that for $A,B \in \Mat_d(\BB)$, if $A_{i,j}=1$ and $B_{j,k}=1$ then $(AB)_{i,k}=1$, we obtain that $N_{F_{\varphi}(N)}=1$ for every morphism $N$ in the category $\catC(X,p)$.    
\end{proof}
\begin{ex}{\rm
The support of $p$ introduced in Example \ref{ex:ABCDU} is $\{\varphi_1,\varphi_2\}$ where
$$
\begin{aligned}
\varphi_1(x)=0,\;\varphi_1(y)=1,\;\varphi_1(z)=1,\,\varphi_1(w)=0,\\
\varphi_2(x)=1,\;\varphi_2(y)=1,\;\varphi_2(z)=0,\,\varphi_2(w)=1.
\end{aligned}
$$
These two maps determine two functors $F_{\varphi_1}:\catC(X,p) \to \catC_{\zz_2}$ and $F_{\varphi_2}:\catC(X,p) \to \catC_{\zz_2}$ such that
$\set{F_{\varphi_1},F_{\varphi_2}}$  
is the
support of $\catC(X,p)$. 
}
\end{ex}
%

%


In Section \ref{susect:simp} we defined the product of $\varphi \in \catsSet(X,\Delta_{\zz_d})$ with $p \in \catsSet(X,D_{\BB}(\Delta_{\zz_d}))$. Similarly, one can define the product of $\varphi$ with the category $\catC(X,p)$ to obtain a new category denoted by $\varphi\cdot \catC(X,p)$ with the same set of objects and whose morphisms are of the form
$$\delta^{(\varphi(x),\varphi(y))}\ast M \in \left(\varphi \cdot \catC(X,p)\right)(x,y)$$ 
where  $M \in \catC(X,p)(x,y)$.
\begin{pro}\label{pro:GroupactCat}
For $\varphi \in \catsSet(X,\Delta_{\zz_d})$ and $p \in \catsSet(X,D_{\BB}(\Delta_{\zz_d}))$, we have 
$$
\varphi \cdot \catC(X,p)=
\catC(X,\varphi\cdot p).
$$
\end{pro}
\begin{proof}
{Consider} non-degenerate $1$-simplices 
$\sigma,\tau \in X_1$ {satisfying}
$$
d_1(\sigma)=x, \; 
d_0(\sigma)=d_1(\tau)=y, \; d_0(\tau)=z.
$$
Let {us write $a=\varphi(x)$, $b=\varphi(y)$, and $c=\varphi(z)$. W}e need to prove that 
$$
(\delta^{(a,b)}\ast p_{\sigma})\circ
(\delta^{(b,c)}\ast p_{\tau})=\delta^{(a,c)}\ast (p_{\sigma}\circ p_{\tau}).
$$ 
Given $i,j \in \zz_d$ we have 
$$
\begin{aligned}
\left((\delta^{(a,b)}\ast p_{\sigma})\circ
(\delta^{(b,c)}\ast p_{\tau})\right)(i,j)&=\sum_{k=0}^{d-1}(\delta^{(a,b)}\ast p_{\sigma})(i,k)
(\delta^{(b,c)}\ast p_{\tau})(k,j)\\
&=\sum_{k=0}^{d-1} p_{\sigma}(i-a,k-b)p_{\tau}(k-b,j-c)\\
&=(p_{\sigma}\circ p_{\tau})(i-a,j-c)\\
&=(\delta^{(a,c)}\ast (p_{\sigma}\circ p_{\tau}))(i,j).
\end{aligned}
$$
In addition, we have $\delta^{(b,a)}\ast M^T=(\delta^{(a,b)}\ast M)^T$.
\end{proof}

\subsection{Strong contextuality for binary outcomes}
\label{sec:strong contextuality for binary outcomes}

In this section we determine when a simplicial distribution $p:X \to D_{\BB}(\Delta_{\zz_2})$ is strongly contextual in terms of the associated category $\catC(X,p)$ and 
give a homotopical characterization  {of}
strongly contextual simplicial distributions $p:X \to D_{}(\Delta_{\zz_2})$.
%
\begin{pro}\label{pro:SCCircle}
Let $C$ be a circle (Definition \ref{def:circle}) and let $C^{(1)}$ be a circle with a single edge $\tau$. For a simplicial distribution $p: C \to D_{\BB}(\Delta_{\zz_d})$ we define $p':C^{(1)} \to D_{\BB}(\Delta_{\zz_d})$ 
by
$$p'_{\tau}=M_1M_2\cdots M_n$$ 
where
$$
M_j=
\begin{cases}
p_{\sigma_j} &  \text{ if } \; i_j=1  \\
p_{\sigma_j}^T & \text{ if } \; i_j=0. 
\end{cases} 
$$
The simplicial distribution $p$ is strongly contextual if and only if $p'$ is strongly contextual. 
\end{pro}
\begin{proof}
The result follows by induction on $n$. For $n=1$ the result is immediate, so assume $n>1$.
Let $Z$ denote the simplicial subset of $C$  generated by $\sigma_1,\sigma_2$.
Let us write $M$ and $N$ for $M_1$ and $M_2$, respectively. 
We replace the simplicial subset $Z\subset C$ by a single edge $\tau$ 
with $d_1(\tau)=d_{i_1}(\sigma_1)$ and $d_0(\tau)=d_{i'_2}(\sigma_2)$, and denote the resulting circle by $C'$. In addition, we define $q:C' \to D_{\BB}(\Delta_{\zz_d})$ on the non-degenerate edges by 
$$
q_{\sigma}=
\begin{cases}
M  N &  \text{ if } \; \sigma=\tau  \\
p_{\sigma} & \text{otherwise.}  
\end{cases} 
$$ 
The support of $p|_{Z}$ is not empty if and only if there exists  $a,b,c\in \zz_d$ such that $M_{a,b}=1$ and $N_{b,c}=1$. 
This is equivalent to saying that there exists $a,c \in \zz_d$ such that $(M N)_{a,c}=1$.
{Therefore} $p$ is strongly contextual if and only if $q$ is strongly contextual. By induction $q$ is strongly contextual if and only if the corresponding $q':C^{(1)} \to D_{\BB}(\Delta_{\zz_d})$ is strongly contextual. Note that $q'=p'$.

%
%
\end{proof}

\begin{cor}\label{cor:LogicalPRbox}
Let $C$ be a circle. A  simplicial distribution $p:C \to D_{\BB}(\Delta_{\zz_2})$ is strongly contextual if and only if $p_{\sigma_j}\in \set{
\begin{pmatrix} 
	1 & 0 \\
	0 & 1\\
	\end{pmatrix}, \begin{pmatrix} 
	0 & 1 \\
	1 & 0 \\
	\end{pmatrix}
}$ for every $1\leq j \leq n$, such that the number of $\begin{pmatrix} 
	0 & 1 \\
	1 & 0 \\
	\end{pmatrix}$ is odd.
\end{cor}
\begin{proof}The only strongly contextual distribution $q\in \catsSet(C^{(1)},D_{\BB}(\Delta_{\zz_2}))$ is the one with 
$$
q_{\tau}=\begin{pmatrix} 
	0 & 1 \\
	1 & 0 \\
	\end{pmatrix}.
$$  
In addition, for   $X,Y\in \Mat_2(\BB)$ we have
\begin{itemize}
\item $X  Y=\begin{pmatrix} 
	0 & 1 \\
	1 & 0 \\
	\end{pmatrix}$ if and only if $\set{X,Y}=\set{\begin{pmatrix} 
	1 & 0 \\
	0 & 1 \\
	\end{pmatrix},\begin{pmatrix} 
	0 & 1 \\
	1 & 0 \\
	\end{pmatrix}}$,
\item $X  Y=\begin{pmatrix} 
	1 & 0 \\
	0 & 1 \\
	\end{pmatrix}$ if and only if $X=Y=\begin{pmatrix} 
	1 & 0 \\
	0 & 1 \\
	\end{pmatrix}$ or $X=Y=\begin{pmatrix} 
	0 & 1 \\
	1 & 0 \\
	\end{pmatrix}$.	
\end{itemize}
{The result {follows from} Proposition \ref{pro:SCCircle}.}

\end{proof}

\begin{lem}\label{lem:NonEmptSup}
Let $\catC$ be a symmetric logical category such that
\begin{itemize}
\item $\catC(x,x)=\{I,U,B,B^T\}$ for every  $x\in \Obj(\catC)$ {(see Example \ref{ex:ABCDU})},
\item $\catC(x,y)=\{U,A,D\}$ or $\{U,B,B^T\}$ for every distinct $x,y\in \Obj(\catC)$.
\end{itemize}
Then $\catC$ has non-empty support.
\end{lem}
\begin{proof}
Fix an object $x$ and define $F:\catC \to \catC_{\zz_2}$ by $F(x)=0$ and for every object $y\neq x$ by
$$
F(y)=
\begin{cases}
0 & \text{ if } \; \catC(x,y)=\{U,B,B^T\},  \\
1 & \text{ if }\; \catC(x,y)=\{U,A,D\}.
\end{cases}
$$
We will prove that $F$ is in the support of $\catC$. By the definition of $F$, for every object $y$ the  $(F(x),F(y))$-entry of all the matrices (morphisms) of $\catC(x,y)$ is $1$. 
Since $\catC$ is symmetric we have that the  $(F(y),F(x))$-entry of all the matrices in $\catC(y,x)$ is $1$. 
Now, 
{given} 
objects $y,z$ 
of $\catC$, if $y=z$, then $(F(y),F(z)) \in \{(0,0),(1,1)\}$ and for every matrix $Z\in \{I,U,B,B^T\}$ we have $Z_{0,0}=Z_{1,1}=1$. Therefore we can assume that $y \neq z$. There are two cases to consider:
\begin{enumerate}
\item
Suppose that $\catC(y,z)=\{U,A,D\}$. If 
$F(y)=0$ and $F(z)=0$, which means that $B \in \catC(x,y)$ and 
$B \in \catC(z,x)$, then we have  
$$
A=BBA \in \catC(z,z).
$$
But $\catC(z,z)=\{I,U,B,B^T\}$. If $F(y)=1$ and $F(z)=1$, which means that $D \in \catC(x,y)$ and 
$A \in \catC(z,x)$, then we have  
$$
A=ADA \in \catC(z,z).
$$
We conclude that $F(y)\neq F(z)$. So the $(F(y),F(z))$-entry of all the matrices in $\catC(y,z)$ is $1$. 

\item
Suppose that $\catC(y,z)=\{U,B,B^T\}$. If 
$F(y)=0$ and $F(z)=1$, which means that $B \in \catC(x,y)$ and 
$D \in \catC(z,x)$, then we have  
$$
D=DBB \in \catC(z,z).
$$
If $F(y)=1$ and $F(z)=0$, which means that $B \in \catC(x,z)$ and 
$D \in \catC(y,x)$, then we have  
$$
D=DBB \in \catC(y,y).
$$
We conclude that $F(y)= F(z)$. So the  $(F(y),F(z))$-entry of  all the matrices in $\catC(y,z)$ is $1$. 
\end{enumerate}
\end{proof}
\begin{cor}\label{cor:CxxADEmpt} 
Let $\catC$ be a symmetric logical category such that
\begin{itemize}
\item $\catC(x,x)\subset \{I,U,B,B^T\}$ for every $x\in \Obj(\catC)$.
\item $\catC(x,y)\subset \{U,A,B,B^T,D\}$ for every distinct objects $x,y\in \Obj(\catC)$, 
\end{itemize}
Then $\catC$ has non-empty support.
\end{cor}
\begin{proof}
{The first condition} implies that for every  $x,y\in \Obj(\catC)$ the set $\catC(x,y)$ does not contain any of the sets 
$$
\set{A,B}, \set{A,B^T}, \set{D,B}, \set{D,B^T}.
$$
This means that $\catC$ is a subcategory of the category described in  Lemma \ref{lem:NonEmptSup}. Therefore $\catC$ has non-empty support. 
\end{proof}
\begin{pro}\label{pro:ABDLSC}
Let $X$ be a $1$-skeletal simplicial set and $p:X\to D_\BB(\Delta_{\ZZ_2})$ be a simplicial distribution such that 
$p_{\tau}\in \{A,B,B^{T},D\}$ for every $\tau\in X_1^\circ$. 
Then the category $\catC(X,p)$ has an empty support if and only if there exits $x \in X_0$ such that $\{A,D\}\subset \catC(X,p)(x,x)$.
\end{pro}
\begin{proof} Through the proof we will write $\catC=\catC(X,p)$. If $A \in \catC(x,x)$, then every $F$ in the support of $\catC$ has to satisfy $F(x)=0$. On the other hand, if $D \in \catC(x,x)$, 
then every $F$ in the support of $\catC$ has to satisfy $F(x)=1$. Thus {if $\set{A,D} \subset \catC(x,x)$ then} the support is empty.

For the converse, suppose that $\{A,D\}\not\subset \catC(x,x)$ for every $x \in X_0$. Let $\catC'$ be the subcategory of $\catC$   generated by the objects $x$ satisfying $\{A,D\}\cap \catC(x,x)=\emptyset$. 
By Corollary \ref{cor:CxxADEmpt} there is a support $F: \catC' \to \catC_{\Delta_{\zz_2}}$ 
for $\catC'$. We extend $F$ to $\catC$ by defining 
$$
F(x)=\begin{cases}
0 & A \in \catC(x,x) \\
1 & D \in \catC(x,x).
\end{cases}
$$
It remains to prove that this extension gives   a support of $\catC $. Given an edge $\tau$ of $X$ with $d_1(\tau)=x$ and $d_0(\tau)=y$, if 
$x,y \in \Obj(\catC')$ then all the matrices in $\catC(x,y)$ are in $\catC'$, so $M_{F(M)}=1$ for every $M \in \catC (x,y)$. Therefore we suppose that  $x\notin \Obj(\catC')$ or $y\notin \Obj(\catC')$.
Assume that $x=y$ (so $x$ is not in $\Obj(\catC')$).
If $A \in \catC(x,x)$ then $F(x)=0$ and $U_{0,0}=B_{0,0}=B^T_{0,0}=A_{0,0}=1$. If $D \in \catC(x,x)$ then $F(x)=1$ and $U_{1,1}=B_{1,1}=B^T_{1,1}=D_{1,1}=1$. Now, suppose that $x \neq y$. We have three cases:
\begin{enumerate}
\item 
$A \in \catC(x,x)$ and  $A \in \catC(y,y)$: In this case $F(M)=(0,0)$ for every $M \in \catC(x,y)$. We have to prove that $D \notin \catC(x,y)$. If $D \in \catC(x,y)$ then $AD=B \in \catC(x,y)$. In addition, $D \in \catC(y,x)$. {Therefore} we obtain that $DB=D\in \catC(y,y)$, but then $\{A,D\} \subset \catC(y,y)$.  Similarly, one can deal with the case where
$D \in \catC(x,x)$ and  $D \in \catC(y,y)$. 

\item
$A \in \catC(x,x)$ and  $D \in \catC(y,y)$: In this case $F(M)=(0,1)$ for every $M \in \catC(x,y)$. We have to prove that $B^T \notin \catC(x,y)$. If $B^T \in \catC(x,y)$ then $AB^T=A \in \catC(x,y)$, so $A \in \catC(y,x)$. Again $AB^T=A \in \catC(y,y)$, but then
$\{A,D\} \subset \catC(y,y)$.  Similarly, one can deal with the case where
$D \in \catC(x,x)$ and  $A \in \catC(y,y)$. 

\item
$A \in \catC(x,x)$ and  $\{A,D\} \cap \catC(y,y)=\emptyset$: In this case $F(M)$ can be $(0,0)$ for every $M \in \catC(x,y)$ or $(0,1)$ for every $M \in \catC(x,y)$. If  $F(M)=(0,0)$ then we have to prove that $D \notin \catC(x,y)$. If $D \in \catC(x,y)$ then as in case (1) we obtain that $D \in \catC(y,y)$. If $F(M)=(0,1)$ then we have to prove that $B^T \notin \catC(x,y)$. If $B^T \in \catC(x,y)$ then as in case (2) we obtain that $A \in \catC(y,y)$.
Similarly, one can deal with the other three analogous cases. 
\end{enumerate}
\end{proof}
Next, we introduce a definition
which will help us to characterize strong contexuality.

\begin{defn}
{\rm
A simplicial distribution $p: \Delta[n] \to D_R(Y)$ has a {\it boundary-extendable support} if its restriction to the boundary induces a surjective map between the supports
$$
\supp(p) \to \supp(p|_{\partial\Delta[n]})
$$
}
\end{defn}
%


\begin{ex}\label{ex:Fotgettttt}
{\rm
A distribution $p \in \catsSet(\Delta[1]
,D_{\BB}(\Delta_{\zz_2}))$ has a boundary-extendable support if $p_{\sigma}$ is {a} {deterministic distribution or} one of the following: 
$$
\begin{pmatrix} 
	1 & 0 \\
	1 & 0 \\
	\end{pmatrix}
,\;\;
\begin{pmatrix} 
	1 & 1 \\
	0 & 0 \\
	\end{pmatrix}
,\;\;
\begin{pmatrix} 
	0 & 1 \\
	0 & 1 \\
	\end{pmatrix}
,\;\;
\begin{pmatrix} 
	0 & 0 \\
	1 & 1 \\
	\end{pmatrix}
,\;\;
\begin{pmatrix} 
	1 & 1 \\
	1 & 1 \\
	\end{pmatrix}.
$$
On the other hand, $p$ does not have a boundary-extendable support if $p_{\sigma}$ is one of the following:   
$$
\begin{pmatrix} 
	1 & 0 \\
	0 & 1 \\
	\end{pmatrix}
,\;\;
\begin{pmatrix} 
	0 & 1 \\
	1 & 0 \\
	\end{pmatrix}
,\;\;
\begin{pmatrix} 
	1 & 1 \\
	1 & 0 \\
	\end{pmatrix}
,\;\;
\begin{pmatrix} 
	1 & 1 \\
	0 & 1 \\
	\end{pmatrix}
,\;\;
\begin{pmatrix} 
	1 & 0 \\
	1 & 1 \\
	\end{pmatrix}
,\;\;
\begin{pmatrix} 
	0 & 1 \\
	1 & 1 \\
	\end{pmatrix}.
$$
}
\end{ex}
\begin{lem}
Every deterministic distribution $\delta^{\psi}$ on the simplicial scenario $(\Delta[n],Y)$ has a boundary-extendable support.
\end{lem}
\begin{proof}
The restriction of $\delta^\psi$ to the boundary is given by $\delta^{\psi|_{\partial\Delta[n]}}$. Therefore the only map in the support of the restriction is $\psi|_{\partial\Delta[n]}$ and it extends to $\psi \in \supp(\delta^\psi)$.
\end{proof}

For a {generating}
simplex $\sigma\in X$ we will write $X-\sigma$ for the simplicial set obtained from $X$ by omitting the simplex $\sigma$ but keeping its boundary.

\begin{pro}\label{pro:OmmitForg}
Let $\sigma\in X_n$ be a {generating} simplex such that $p_\sigma$ has a boundary-extendable support.
Then $p$ is strongly contextual if and only if $p|_{X-\sigma}$ is strongly contextual.
\end{pro}
\begin{proof}
If $\psi \in \supp(p)$ then $\psi|_{X-\sigma} \in \supp(p|_{X-\sigma})$. On the other hand, if $\varphi \in \supp(p|_{X-\sigma})$, then $\varphi|_{\partial \sigma}\in \supp(p|_{\partial \sigma})$ can be extended to a map in $\supp(p_\sigma)$. This give us a map in $\supp(p)$. 
\end{proof}

\begin{pro}\label{Pro:CharofLSC}
Let $X$ be a $1$-skeletal simplicial set and $p:X\to D_\BB(\Delta_{\ZZ_2})$ be a simplicial distribution. Then $p$ is strongly contextual if and only if there exits $x \in X_0$ such that $\{A,D\}\subset \catC(X,p)(x,x)$ or $\begin{pmatrix} 
	0 & 1\\
	1 & 0 \\
	\end{pmatrix}
	\in \catC(X,p)(x,x)$.
\end{pro}
\begin{proof}
If for some $x \in X_0$ we have $\{A,D\}\subset \catC(X,p)(x,x)$ or $\begin{pmatrix} 
	0 & 1\\
	1 & 0 \\
	\end{pmatrix}
	\in \catC(X,p)(x,x)$, then $\catC(X,p)$ has an empty support. By Proposition \ref{pr:DisSCCatSC}  $p$ is strongly contextual. 
	
	Conversely, suppose that $p$ is strongly contextual. Let $\tilde{X}$ be the simplicial subset of $X$ obtained by omitting every {generating} simplex $\sigma$ whenever $p_{\sigma}$ has a boundary-extendable support. By Proposition \ref{pro:OmmitForg}  $p|_{\tilde{X}}$ is strongly contextual.	
Then for every edge $\tau$ of $\tilde{X}$ the distribution $p_{\tau}$ is one of the distributions of Example \ref{ex:Fotgettttt} that does not have a boundary-extendable support. Fix an edge $\tau$ of $\tilde{X}$. If 
we have
$$ 
p_\tau =\begin{pmatrix} 
	1 & 0 \\
	0 & 1 \\
	\end{pmatrix}
$$
then $p_{\tau}$  vanishes on $(\Delta_{\zz_2})_1-s_0((\Delta_{\zz_2})_0)=
\zz_2\times \zz_2-s_0(\zz_2)=\{(0,1),(1,0)\}$. 
By Corollary \ref{cor:CollaDostCha} the simplicial distribution $p|_{\tilde{X}}$ factors through the collapsing map $\pi: \tilde{X} \to \tilde{X}^{s^0,\tau}$ that comes from the following pushout:
$$
\begin{tikzcd}[column sep =huge, row sep =large]
\Delta[1]
\arrow[r,"\tau"]
 \arrow[d,"s^0"'] & \tilde X  
 \arrow[d,"\pi"] \\
\Delta[0]
 \arrow[r] & \tilde{X}^{s^0,\tau} 
\end{tikzcd}
$$
This means that there is a simplicial distribution $p'\in \catsSet(\tilde{X}^{s^0,\tau},D_{\BB}(\Delta_
{\zz_2}))$ such that $\pi^\ast(p')=p|_{\tilde{X}}$. According to Theorem \ref{thm:CollapThm} the distribution $p'$ is strongly contextual. We collapse all edges with such a distribution to {obtain} a strongly contextual distribution $p''$ on the collapsed space $\tilde{X}'$. If we have an edge $\tau$ with 
%
$$ p''_\tau=
\begin{pmatrix} 
	0 & 1 \\
	1 & 0 \\
\end{pmatrix}
$$
then by choosing a simplicial map $\varphi: \tilde{X}' \to \zz_2$ with $\varphi_{d_1(\tau)}=0$ and 
 $\varphi_{d_0(\tau)}=1$, we obtain that $(\delta^{\varphi} \cdot p'')_{\tau}= 
\begin{pmatrix} 
	1 & 0 \\
	0 & 1 \\
\end{pmatrix}
$. Moreover, for every edge $\sigma$ such that
$p''_{\sigma}$ is one of the following distributions $A,B,B^T,D$, the distribution $(\delta^\varphi \cdot p'')_{\sigma}$ is also one of the distributions $A,B,B^T,D$.
By Proposition \ref{pro:SCGroupaction} the distribution $\delta^{\varphi}\cdot p''$ is strongly contextual. In this case, we can collapse the edge $\tau$ as before, and still have a strongly contextual distribution.
If by this process we obtain a circle with a single  edge on which we have the distribution $\begin{pmatrix} 
	1 & 0 \\
	0 & 1 \\
\end{pmatrix}$, then it is easy to see that  
the  simplicial distribution $q:W \to D_{\BB}(\Delta_{\zz_2})$ obtained by omitting this kind of circles is still strongly contextual.
We have the following two cases:
\begin{enumerate}
 
\item 
We have a circle with a single edge $\sigma$ in $W$ such that $q_{\sigma}=\begin{pmatrix} 
	0 & 1 \\
	1 & 0 \\
\end{pmatrix}$, which is the only strongly contextual distribution on such a circle. 
As a consequence of Theorem \ref{thm:CollapThm} and  Proposition \ref{pro:SCGroupaction} there exists a circle $C \subset X$ such that $p|_{C}$ is strongly contextual. So by Proposition \ref{pro:SCCircle} we conclude that for every vertex $x$ in $C$ we have   
$$
\begin{pmatrix} 
	0 & 1 \\
	1 & 0 \\
\end{pmatrix} \in \catC(X,p)(x,x).
$$
\item In the case that there is no circle with a single edge with the distribution $\begin{pmatrix} 
	0 & 1 \\
	1 & 0 \\
\end{pmatrix}$, our induced distribution $q:W \to D_{\BB}(\Delta_{\zz_2})$ satisfies the conditions of Proposition \ref{pro:ABDLSC}. Therefore there exists $w\in W_0$ such that $A,D \in \catC(W,q)(w,w)$. {Note that the collapsed distribution  {$\begin{pmatrix} 
	1 & 0 \\
	0 & 1 \\
\end{pmatrix}$} {serves as} the identity 
{element}
in the product of matrices. In addition, we have}   
$$
\delta^{(0,0)}\ast A=A, \;\;  \delta^{(0,0)}\ast D=D, \;\; \delta^{(1,1)}\ast A=D, \;\;  \delta^{(1,1)}\ast D=A.
$$
%
Therefore using Proposition \ref{pro:GroupactCat} we conclude that 
$A,D \in \catC(X,p)(x,x)$ for some vertex $x \in X_0$. 
\end{enumerate} 
\end{proof}
We will use Proposition \ref{Pro:CharofLSC}  to give a  characterization for   
strong contextuality when the semiring is $\rr_{\geq 0}$.

\begin{pro}\label{pr:SCandLSC}
A simplicial distribution $p:X\to D(Y)$ is strongly contextual if and only if $\Pi_*(p)$ is strongly contextual {(see the map {in} (\ref{eq:PiRealToBool})).} 
\end{pro}
\begin{proof}
Follows directly from the definition of strong contextuality.
\end{proof}

\begin{lem}\label{lem:ADst}Let $p:\Delta[1]\to D(\Delta_{\ZZ_2})$ be a simplicial distribution. 
We have the following: 
\begin{enumerate}
\item If $\Pi_*(p)_{\sigma^{01}}=A$ then $p_{\sigma^1}(0)+p_{\sigma^0}(0)>1$.
\item If $\Pi_*(p)_{\sigma^{01}}=D$ then $p_{\sigma^1}(0)+p_{\sigma^0}(0)<1$.
\item If $\Pi_*(p)_{\sigma^{01}}=B$ then $p_{\sigma^0}(0)>p_{\sigma^1}(0)$.
\item If $\Pi_*(p)_{\sigma^{01}}=B^T$ then $p_{\sigma^0}(0)<p_{\sigma^1}(0)$.
 
\item If $\Pi_*(p)_{\sigma^{01}}=\begin{pmatrix} 
	1 & 0 \\
	0 & 1 \\
	\end{pmatrix}$ then $p_{\sigma^0}(0)=p_{\sigma^1}(0)$.
\item If $\Pi_*(p)_{\sigma^{01}}=\begin{pmatrix} 
	0 & 1 \\
	1 & 0 \\
	\end{pmatrix}$ then $p_{\sigma^0}(0)=1-p_{\sigma^1}(0)$.
\end{enumerate}
\end{lem}
\begin{proof}
We will prove only 
 part (1),  other parts are similar. Suppose that $p_{\sigma^{01}}=\begin{pmatrix} 
	p_1 & p_2 \\
	p_3 & p_4 \\
	\end{pmatrix}$, then we have 
$$
p_{\sigma^1}(0)=
D(d_0)(p_{\sigma^{01}})(0)=p_{\sigma^{01}}(0,0)+p_{\sigma^{01}}(1,0)=p_1+p_3.
$$ 
Similarly, $p_{\sigma^0}(0)=p_1+p_2$.  Since $\Pi_*(p)_{\sigma^{01}}=A$ we obtain that $p_4=0$. Therefore we have 
$$
p_{\sigma^1}(0)+p_{\sigma^0}(0)=p_1+p_3+p_1+p_2
=p_1+1>1.
$$

\end{proof}

\begin{cor}\label{cor:BoleanLift}
Let $X$ be a $1$-skeletal simplicial set and  $p: X\to D(\Delta_{\ZZ_2})$ be a simplicial distribution, we have the following:
\begin{enumerate}
\item If $A \in \catC(X,\Pi_*(p))(x,y)$ then $p_x(0)+p_y(0)>1$. 
\item If $D \in \catC(X,\Pi_*(p))(x,y)$ then $p_x(0)+p_y(0)<1$.
\item If $B \in \catC(X,\Pi_*(p))(x,y)$ then $p_x(0) >p_y(0)$.
\item If $B^T \in \catC(X,\Pi_*(p))(x,y)$ then $p_x(0)<p_y(0)$.
\item If $\begin{pmatrix} 
	1 & 0 \\
	0 & 1 \\
	\end{pmatrix}
\in \catC(X,\Pi_*(p))(x,y)$	then $p_{x}(0)=p_{y}(0)$.
\item If $\begin{pmatrix} 
	0 & 1 \\
	1 & 0 \\
	\end{pmatrix}
\in \catC(X,\Pi_*(p))(x,y)$	then $p_{x}(0)=1-p_{y}(0)$.
\end{enumerate}
\end{cor}
\begin{proof}
Let $M \in \catC(X,\Pi_*(p))(x,y)$. By Diagram (\ref{dia:PiPi}) there exists $q \in \catC(X,p)(x,y)$ such that $\Pi_\ast(q)=M$. Then by Lemmas \ref{lem:Impo} and \ref{lem:ADst} we obtain the result.
\end{proof}

%
%

%
%
%

%
\begin{thm}\label{thm:SC1skelz2}
Let $X$ be a $1$-skeletal simplicial set.
A simplicial distribution $p:X\to D(\Delta_{\ZZ_2})$
is strongly contextual if and only if there exists a circle $C\subset X$ such that $p|_{C}$ is a PR box (see Definition \ref{def:PR} and Remark \ref{rem:PRBox}).
\end{thm}
\begin{proof}
Suppose that there exists a circle $C\subset X$ such that the restriction $p|_{C}$ is a PR box. This  is a strongly contextual distribution according to Corollary \ref{ex:PRBox}. Therefore $p$ is strongly contextual, because if $\varphi \in \supp(p)$, then $\varphi|_C \in \supp(p|_{C})$.

Conversely, suppose that $p$ is strongly contextual. By Proposition \ref{pr:SCandLSC}  $\Pi_\ast(p)$ is strongly contextual. 
Therefore according to Proposition \ref{Pro:CharofLSC} there exists a vertex $x$ of $X$ such that $\{A,D\}\subset \catC(X,\Pi_\ast(p))(x,x)$ or $\begin{pmatrix} 
	0 & 1\\
	1 & 0 \\
	\end{pmatrix}
	\in \catC(X,\Pi_\ast(p))(x,x)$. 
	According to Corollary \ref{cor:BoleanLift}, $A \in \catC(X,\Pi_\ast(p))(x,x)$ implies that $p_x(0) + p_x(0)>1$. On the other hand, $D \in \catC(X,\Pi_\ast(p))(x,x)$ implies that $p_x(0) +p_x(0)<1$. We conclude that $\begin{pmatrix} 
	0 & 1\\
	1 & 0 \\
	\end{pmatrix}
	\in \catC(X,\Pi_\ast(p))(x,x)$, which is the unique strongly contextual distribution on the circle with a single edge. 
	By Proposition \ref{pro:SCCircle} we conclude that there exists a circle $C \subset X$ such that $\Pi_\ast(p)|_C$ is strongly contextual. Then by Corollary  \ref{cor:LogicalPRbox} we obtain that  
$\pi_\ast(p)_{\sigma}\in \set{
\begin{pmatrix} 
	1 & 0 \\
	0 & 1\\
	\end{pmatrix}, \begin{pmatrix} 
	0 & 1 \\
	1 & 0 \\
	\end{pmatrix}
}$ for every edge $\sigma$ of the circle $C$ such that the number of $\begin{pmatrix} 
	0 & 1 \\
	1 & 0 \\
	\end{pmatrix}$ is odd. Finally, given a vertex $x$ of $C$, by parts (5) and (6) of {Corollary \ref{cor:BoleanLift}} we obtain that $p_{x}(0)=1-p_{x}(0)$, which {gives} $p_{x}(0)=\frac{1}{2}$. Therefore $p$ is a PR box.     
\end{proof}

\begin{cor} 
\label{cor:homotopical characterization of strong contextuality}
Let $X$ be a $1$-skeletal simplicial set.
A simplicial distribution $p:\Delta[0]\ast X\to D(N\ZZ_2)$ 
is strongly contextual if and only if there exists a circle $C\subset X$ 
such that $p|_{C}=\delta^{\varphi}$ for some simplicial set map $\varphi: C \to N\ZZ_2$ that is not null-homotopic.
\end{cor}
\begin{proof}
The simplicial distribution  $p:\Delta[0]\ast X\to D(N\ZZ_2)$ 
is strongly contextual if and only if its transpose $\tilde{p}: X \to {D(\Dec^0(N\zz_2)) \cong D(\Delta_{\ZZ_2})}$ 
 under the adjunction in (\ref{eq:Cone vs declage adjunction}) is strongly contextual. By Theorem \ref{thm:SC1skelz2} there exists a circle $C \subset X$ such that $\tilde{p}|_{C}$ is a PR box. As we have seen  in the proof of Corollary \ref{ex:PRBox} $D(d_0)\circ \tilde{p}|_C$ is a deterministic distribution $\delta^{\varphi}$ for some simplicial set map $\varphi: C \to N\ZZ_2$ that is not null-homotopic. By Lemma \ref{lem:transcomm} we have $D(d_0)\circ \tilde{p}|_C=p|_C$.
\end{proof}


\bibliography{bib.bib}
\bibliographystyle{ieeetr}

\end{document}